\documentclass[12pt, reqno]{amsart}

\usepackage{amsmath,amssymb,amscd,mathrsfs, mdwlist}
\usepackage{graphics,verbatim, tikz}
\allowdisplaybreaks

\usetikzlibrary{matrix,arrows,patterns, snakes}
\usepackage{amstext,amsbsy}
\usepackage{amsxtra}
\usepackage{amsthm}
\usepackage{amsfonts}
\usepackage{graphicx}%
\usepackage{eucal}
\usepackage{color}
\usepackage{multirow}
\usepackage[enableskew,vcentermath]{youngtab}

\newtheorem{thm}{Theorem}[section]
\theoremstyle{plain}
\newtheorem*{thm*}{{\bf Theorem}}
\newtheorem{lem}[thm]{Lemma}

\newtheorem{prop}[thm]{Proposition}

\theoremstyle{definition}
\newtheorem{dfn}[thm]{Definition}

\theoremstyle{remark}

\definecolor{A}{rgb}{.75,1,.75}

\numberwithin{equation}{section}

\usepackage[all,knot]{xy}
\xyoption{arc}

\usetikzlibrary{decorations.markings}
\tikzset{->->-/.style={decoration={
  markings,
  mark=at position .25 with {\arrow[scale=1.5,black]{>}}, mark=at position .85 with {\arrow[scale=1.5,black]{>}}},postaction={decorate}}}
\tikzset{->-/.style={decoration={
  markings,
  mark=at position .55 with {\arrow[scale=1.5,black]{>}}},postaction={decorate}}}
\tikzset{->/.style={decoration={
  markings,
  mark=at position .99 with {\arrow[scale=1.5,black]{>}}},postaction={decorate}}}
\tikzset{<-/.style={decoration={
  markings,
  mark=at position .10 with {\arrow[scale=1.5,black]{<}}},postaction={decorate}}}
\tikzset{-<-/.style={decoration={
  markings,
  mark=at position .55 with {\arrow[scale=1.5,black]{<}}},postaction={decorate}}}
\tikzset{-<-<-/.style={decoration={
  markings,
  mark=at position .25 with {\arrow[scale=1.5,black]{<}}, mark=at position .85 with {\arrow[scale=1.5,black]{<}}},postaction={decorate}}}

\newcommand{\nesc}[4]{
\pgfmathsetmacro{\xscale}{{#3}}
\pgfmathsetmacro{\yscale}{{#4}}
\pgfmathsetmacro{\xcenter}{{#1}}
\pgfmathsetmacro{\ycenter}{{#2}}
	\pgfmathparse{-.5*\xscale+\xcenter}		\let\Xa\pgfmathresult
    \pgfmathparse{-1*\yscale+\ycenter}		\let\Ya\pgfmathresult
    \coordinate (A) at (\Xa,\Ya);
	\pgfmathparse{-.5*\xscale+\xcenter}		\let\Xb\pgfmathresult
    \pgfmathparse{0+\ycenter}		\let\Yb\pgfmathresult
    \coordinate (B) at (\Xb,\Yb);
	\pgfmathparse{-.25*\xscale+\xcenter}		\let\Xc\pgfmathresult
    \pgfmathparse{.5*\yscale+\ycenter}		\let\Yc\pgfmathresult
    \coordinate (C) at (\Xc,\Yc);
	\pgfmathparse{0+\xcenter}		\let\Xd\pgfmathresult
    \pgfmathparse{0+\ycenter}		\let\Yd\pgfmathresult
    \coordinate (D) at (\Xd,\Yd);
	\pgfmathparse{.25*\xscale+\xcenter}		\let\Xe\pgfmathresult
    \pgfmathparse{-.5*\yscale+\ycenter}		\let\Ye\pgfmathresult
    \coordinate (E) at (\Xe,\Ye);
	\pgfmathparse{.5*\xscale+\xcenter}		\let\Xf\pgfmathresult
    \pgfmathparse{0*\yscale+\ycenter}		\let\Yf\pgfmathresult
    \coordinate (F) at (\Xf,\Yf);
	\pgfmathparse{.5*\xscale+\xcenter}		\let\Xg\pgfmathresult
    \pgfmathparse{1*\yscale+\ycenter}		\let\Yg\pgfmathresult
    \coordinate (G) at (\Xg,\Yg);
\draw[line width=1pt] (A) .. controls (B) and (C) .. (D) .. controls (E) and (F)  .. (G);
}
\newcommand{\dcap}[4]{
\pgfmathsetmacro{\xscale}{{#3}}
\pgfmathsetmacro{\yscale}{{#4}}
\pgfmathsetmacro{\xleft}{{#1}}
\pgfmathsetmacro{\ybottom}{{#2}}
	\pgfmathparse{\xleft}		\let\Xa\pgfmathresult
    \pgfmathparse{\ybottom}		\let\Ya\pgfmathresult
    \coordinate (A) at (\Xa,\Ya);
	\pgfmathparse{\xleft}		\let\Xb\pgfmathresult
    \pgfmathparse{\ybottom+\yscale}		\let\Yb\pgfmathresult
    \coordinate (B) at (\Xb,\Yb);
	\pgfmathparse{\xleft+\xscale}		\let\Xc\pgfmathresult
    \pgfmathparse{\ybottom+\yscale}		\let\Yc\pgfmathresult
    \coordinate (C) at (\Xc,\Yc);
	\pgfmathparse{\xleft+\xscale}		\let\Xd\pgfmathresult
    \pgfmathparse{\ybottom}		\let\Yd\pgfmathresult
    \coordinate (D) at (\Xd,\Yd);
\draw[line width=1pt] (A) .. controls (B) and (C) .. (D);
}
\newcommand{\cwcap}[4]{
\pgfmathsetmacro{\xscale}{{#3}}
\pgfmathsetmacro{\yscale}{{#4}}
\pgfmathsetmacro{\xleft}{{#1}}
\pgfmathsetmacro{\ybottom}{{#2}}
	\pgfmathparse{\xleft}		\let\Xa\pgfmathresult
    \pgfmathparse{\ybottom}		\let\Ya\pgfmathresult
    \coordinate (A) at (\Xa,\Ya);
	\pgfmathparse{\xleft}		\let\Xb\pgfmathresult
    \pgfmathparse{\ybottom+\yscale}		\let\Yb\pgfmathresult
    \coordinate (B) at (\Xb,\Yb);
	\pgfmathparse{\xleft+\xscale}		\let\Xc\pgfmathresult
    \pgfmathparse{\ybottom+\yscale}		\let\Yc\pgfmathresult
    \coordinate (C) at (\Xc,\Yc);
	\pgfmathparse{\xleft+\xscale}		\let\Xd\pgfmathresult
    \pgfmathparse{\ybottom}		\let\Yd\pgfmathresult
    \coordinate (D) at (\Xd,\Yd);
\draw[line width=1pt, ->-] (A) .. controls (B) and (C) .. (D);
}
\newcommand{\ccwcap}[4]{
\pgfmathsetmacro{\xscale}{{#3}}
\pgfmathsetmacro{\yscale}{{#4}}
\pgfmathsetmacro{\xleft}{{#1}}
\pgfmathsetmacro{\ybottom}{{#2}}
	\pgfmathparse{\xleft}		\let\Xa\pgfmathresult
    \pgfmathparse{\ybottom}		\let\Ya\pgfmathresult
    \coordinate (A) at (\Xa,\Ya);
	\pgfmathparse{\xleft}		\let\Xb\pgfmathresult
    \pgfmathparse{\ybottom+\yscale}		\let\Yb\pgfmathresult
    \coordinate (B) at (\Xb,\Yb);
	\pgfmathparse{\xleft+\xscale}		\let\Xc\pgfmathresult
    \pgfmathparse{\ybottom+\yscale}		\let\Yc\pgfmathresult
    \coordinate (C) at (\Xc,\Yc);
	\pgfmathparse{\xleft+\xscale}		\let\Xd\pgfmathresult
    \pgfmathparse{\ybottom}		\let\Yd\pgfmathresult
    \coordinate (D) at (\Xd,\Yd);
\draw[line width=1pt, -<-] (A) .. controls (B) and (C) .. (D);
}
\newcommand{\dcup}[4]{
\pgfmathsetmacro{\xscale}{{#3}}
\pgfmathsetmacro{\yscale}{{#4}}
\pgfmathsetmacro{\xleft}{{#1}}
\pgfmathsetmacro{\ybottom}{{#2}}
	\pgfmathparse{\xleft}		\let\Xa\pgfmathresult
    \pgfmathparse{\ybottom+\yscale}		\let\Ya\pgfmathresult
    \coordinate (A) at (\Xa,\Ya);
	\pgfmathparse{\xleft}		\let\Xb\pgfmathresult
    \pgfmathparse{\ybottom}		\let\Yb\pgfmathresult
    \coordinate (B) at (\Xb,\Yb);
	\pgfmathparse{\xleft+\xscale}		\let\Xc\pgfmathresult
    \pgfmathparse{\ybottom}		\let\Yc\pgfmathresult
    \coordinate (C) at (\Xc,\Yc);
	\pgfmathparse{\xleft+\xscale}		\let\Xd\pgfmathresult
    \pgfmathparse{\ybottom+\yscale}		\let\Yd\pgfmathresult
    \coordinate (D) at (\Xd,\Yd);
\draw[line width=1pt] (A) .. controls (B) and (C) .. (D);
}
\newcommand{\cwcup}[4]{
\pgfmathsetmacro{\xscale}{{#3}}
\pgfmathsetmacro{\yscale}{{#4}}
\pgfmathsetmacro{\xleft}{{#1}}
\pgfmathsetmacro{\ybottom}{{#2}}
	\pgfmathparse{\xleft}		\let\Xa\pgfmathresult
    \pgfmathparse{\ybottom+\yscale}		\let\Ya\pgfmathresult
    \coordinate (A) at (\Xa,\Ya);
	\pgfmathparse{\xleft}		\let\Xb\pgfmathresult
    \pgfmathparse{\ybottom}		\let\Yb\pgfmathresult
    \coordinate (B) at (\Xb,\Yb);
	\pgfmathparse{\xleft+\xscale}		\let\Xc\pgfmathresult
    \pgfmathparse{\ybottom}		\let\Yc\pgfmathresult
    \coordinate (C) at (\Xc,\Yc);
	\pgfmathparse{\xleft+\xscale}		\let\Xd\pgfmathresult
    \pgfmathparse{\ybottom+\yscale}		\let\Yd\pgfmathresult
    \coordinate (D) at (\Xd,\Yd);
\draw[line width=1pt, -<-] (A) .. controls (B) and (C) .. (D);
}
\newcommand{\ccwcup}[4]{
\pgfmathsetmacro{\xscale}{{#3}}
\pgfmathsetmacro{\yscale}{{#4}}
\pgfmathsetmacro{\xleft}{{#1}}
\pgfmathsetmacro{\ybottom}{{#2}}
	\pgfmathparse{\xleft}		\let\Xa\pgfmathresult
    \pgfmathparse{\ybottom+\yscale}		\let\Ya\pgfmathresult
    \coordinate (A) at (\Xa,\Ya);
	\pgfmathparse{\xleft}		\let\Xb\pgfmathresult
    \pgfmathparse{\ybottom}		\let\Yb\pgfmathresult
    \coordinate (B) at (\Xb,\Yb);
	\pgfmathparse{\xleft+\xscale}		\let\Xc\pgfmathresult
    \pgfmathparse{\ybottom}		\let\Yc\pgfmathresult
    \coordinate (C) at (\Xc,\Yc);
	\pgfmathparse{\xleft+\xscale}		\let\Xd\pgfmathresult
    \pgfmathparse{\ybottom+\yscale}		\let\Yd\pgfmathresult
    \coordinate (D) at (\Xd,\Yd);
\draw[line width=1pt, ->-] (A) .. controls (B) and (C) .. (D);
}

\newcommand{\ids}[5]{
\pgfmathsetmacro{\xl}{{#1}}
\pgfmathsetmacro{\yb}{{#2}}
\pgfmathsetmacro{\num}{{#5}}
\pgfmathsetmacro{\yheight}{{#4}}
\pgfmathsetmacro{\yt}{\yb+\yheight};
\pgfmathsetmacro{\xscale}{{#3}};
	\foreach \x in {1,...,\num}{
		\pgfmathparse{\xl+(\x-1)*\xscale}\let\xc\pgfmathresult;
		\draw[line width=1pt] (\xc,\yb) -- (\xc, \yt);
	}
}
\newcommand{\idsup}[5]{
\pgfmathsetmacro{\xl}{{#1}}
\pgfmathsetmacro{\yb}{{#2}}
\pgfmathsetmacro{\num}{{#5}}
\pgfmathsetmacro{\yheight}{{#4}}
\pgfmathsetmacro{\yt}{\yb+\yheight};
\pgfmathsetmacro{\xscale}{{#3}};
	\foreach \x in {1,...,\num}{
		\pgfmathparse{\xl+(\x-1)*\xscale}\let\xc\pgfmathresult
		\draw[line width=1pt,->-] (\xc,\yb) -- (\xc, \yt);
	}
}
\newcommand{\idsdown}[5]{
\pgfmathsetmacro{\xl}{{#1}}
\pgfmathsetmacro{\yb}{{#2}}
\pgfmathsetmacro{\num}{{#5}}
\pgfmathsetmacro{\yheight}{{#4}}
\pgfmathsetmacro{\yt}{\yb+\yheight};
\pgfmathsetmacro{\xscale}{{#3}};
	\foreach \x in {1,...,\num}{
		\pgfmathparse{\xl+(\x-1)*\xscale}\let\xc\pgfmathresult
		\draw[line width=1pt,-<-] (\xc,\yb) -- (\xc, \yt);
	}
}

\newcommand{\rcross}[4]{
\tikzstyle{cross line}=[preaction={draw=white, -,line width=10pt}];
\pgfmathsetmacro{\xl}{{#1}}
\pgfmathsetmacro{\yb}{{#2}}
\pgfmathsetmacro{\yheight}{{#3}}
\pgfmathsetmacro{\xscale}{{#4}}
\pgfmathsetmacro{\yt}{\yb+\yheight};
\pgfmathsetmacro{\xr}{\xscale+\xl};
\pgfmathsetmacro{\yc}{\yb+\yheight/2};
\pgfmathsetmacro{\xc}{\xscale/2+\xl};

	\pgfmathparse{\xl}		\let\Xa\pgfmathresult
    \pgfmathparse{\yt}		\let\Ya\pgfmathresult
    \coordinate (A) at (\Xa,\Ya);
	\pgfmathparse{\xl+(0.3*\xscale)}		\let\Xaa\pgfmathresult
    \pgfmathparse{\yb+(0.9*\yheight)}		\let\Yaa\pgfmathresult
    \coordinate (A') at (\Xaa,\Yaa);
	\pgfmathparse{\xl}		\let\Xb\pgfmathresult
    \pgfmathparse{\yb}		\let\Yb\pgfmathresult
    \coordinate (B) at (\Xb,\Yb);
	\pgfmathparse{\xl+(0.3*\xscale)}		\let\Xbb\pgfmathresult
    \pgfmathparse{\yb+(0.1*\yheight)}		\let\Ybb\pgfmathresult
    \coordinate (B') at (\Xbb,\Ybb);
	\pgfmathparse{\xr}		\let\Xc\pgfmathresult
    \pgfmathparse{\yb}		\let\Yc\pgfmathresult
    \coordinate (C) at (\Xc,\Yc);
	\pgfmathparse{\xl+(0.7*\xscale)}		\let\Xcc\pgfmathresult
    \pgfmathparse{\yb+(0.1*\yheight)}		\let\Ycc\pgfmathresult
    \coordinate (C') at (\Xcc,\Ycc);
	\pgfmathparse{\xr}		\let\Xd\pgfmathresult
    \pgfmathparse{\yt}		\let\Yd\pgfmathresult
    \coordinate (D) at (\Xd,\Yd);
	\pgfmathparse{\xl+(0.7*\xscale)}		\let\Xdd\pgfmathresult
    \pgfmathparse{\yb+(0.9*\yheight)}		\let\Ydd\pgfmathresult
    \coordinate (D') at (\Xdd,\Ydd);
	\pgfmathparse{\xc}		\let\Xe\pgfmathresult
    \pgfmathparse{\yc}		\let\Ye\pgfmathresult
    \coordinate (E) at (\Xe,\Ye);

\draw[line width=1pt] (A) .. controls (B') and (D') .. (C);
\draw (\xc, \yc) node[shape=circle, fill=white] {};
\draw[line width=1pt] (B) .. controls (A') and (C') .. (D);
}

\newcommand{\rcrossup}[4]{
\tikzstyle{cross line}=[preaction={draw=white, -,line width=10pt}];
\pgfmathsetmacro{\xl}{{#1}}
\pgfmathsetmacro{\yb}{{#2}}
\pgfmathsetmacro{\yheight}{{#3}}
\pgfmathsetmacro{\xscale}{{#4}}
\pgfmathsetmacro{\yt}{\yb+\yheight};
\pgfmathsetmacro{\xr}{\xscale+\xl};
\pgfmathsetmacro{\yc}{\yb+\yheight/2};
\pgfmathsetmacro{\xc}{\xscale/2+\xl};

	\pgfmathparse{\xl}		\let\Xa\pgfmathresult
    \pgfmathparse{\yt}		\let\Ya\pgfmathresult
    \coordinate (A) at (\Xa,\Ya);
	\pgfmathparse{\xl+(0.3*\xscale)}		\let\Xaa\pgfmathresult
    \pgfmathparse{\yb+(0.9*\yheight)}		\let\Yaa\pgfmathresult
    \coordinate (A') at (\Xaa,\Yaa);
	\pgfmathparse{\xl}		\let\Xb\pgfmathresult
    \pgfmathparse{\yb}		\let\Yb\pgfmathresult
    \coordinate (B) at (\Xb,\Yb);
	\pgfmathparse{\xl+(0.3*\xscale)}		\let\Xbb\pgfmathresult
    \pgfmathparse{\yb+(0.1*\yheight)}		\let\Ybb\pgfmathresult
    \coordinate (B') at (\Xbb,\Ybb);
	\pgfmathparse{\xr}		\let\Xc\pgfmathresult
    \pgfmathparse{\yb}		\let\Yc\pgfmathresult
    \coordinate (C) at (\Xc,\Yc);
	\pgfmathparse{\xl+(0.7*\xscale)}		\let\Xcc\pgfmathresult
    \pgfmathparse{\yb+(0.1*\yheight)}		\let\Ycc\pgfmathresult
    \coordinate (C') at (\Xcc,\Ycc);
	\pgfmathparse{\xr}		\let\Xd\pgfmathresult
    \pgfmathparse{\yt}		\let\Yd\pgfmathresult
    \coordinate (D) at (\Xd,\Yd);
	\pgfmathparse{\xl+(0.7*\xscale)}		\let\Xdd\pgfmathresult
    \pgfmathparse{\yb+(0.9*\yheight)}		\let\Ydd\pgfmathresult
    \coordinate (D') at (\Xdd,\Ydd);
	\pgfmathparse{\xc}		\let\Xe\pgfmathresult
    \pgfmathparse{\yc}		\let\Ye\pgfmathresult
    \coordinate (E) at (\Xe,\Ye);

\draw[line width=1pt, <-] (A) .. controls (B') and (D') .. (C);
\draw (\xc, \yc) node[shape=circle, fill=white] {};
\draw[line width=1pt, ->] (B) .. controls (A') and (C') .. (D);
}
\newcommand{\NESE}[4]{
\tikzstyle{cross line}=[preaction={draw=white, -,line width=10pt}];
\pgfmathsetmacro{\xl}{{#1}}
\pgfmathsetmacro{\yb}{{#2}}
\pgfmathsetmacro{\yheight}{{#3}}
\pgfmathsetmacro{\xscale}{{#4}}
\pgfmathsetmacro{\yt}{\yb+\yheight};
\pgfmathsetmacro{\xr}{\xscale+\xl};
\pgfmathsetmacro{\yc}{\yb+\yheight/2};
\pgfmathsetmacro{\xc}{\xscale/2+\xl};

	\pgfmathparse{\xl}		\let\Xa\pgfmathresult
    \pgfmathparse{\yt}		\let\Ya\pgfmathresult
    \coordinate (A) at (\Xa,\Ya);
	\pgfmathparse{\xl+(0.3*\xscale)}		\let\Xaa\pgfmathresult
    \pgfmathparse{\yb+(0.9*\yheight)}		\let\Yaa\pgfmathresult
    \coordinate (A') at (\Xaa,\Yaa);
	\pgfmathparse{\xl}		\let\Xb\pgfmathresult
    \pgfmathparse{\yb}		\let\Yb\pgfmathresult
    \coordinate (B) at (\Xb,\Yb);
	\pgfmathparse{\xl+(0.3*\xscale)}		\let\Xbb\pgfmathresult
    \pgfmathparse{\yb+(0.1*\yheight)}		\let\Ybb\pgfmathresult
    \coordinate (B') at (\Xbb,\Ybb);
	\pgfmathparse{\xr}		\let\Xc\pgfmathresult
    \pgfmathparse{\yb}		\let\Yc\pgfmathresult
    \coordinate (C) at (\Xc,\Yc);
	\pgfmathparse{\xl+(0.7*\xscale)}		\let\Xcc\pgfmathresult
    \pgfmathparse{\yb+(0.1*\yheight)}		\let\Ycc\pgfmathresult
    \coordinate (C') at (\Xcc,\Ycc);
	\pgfmathparse{\xr}		\let\Xd\pgfmathresult
    \pgfmathparse{\yt}		\let\Yd\pgfmathresult
    \coordinate (D) at (\Xd,\Yd);
	\pgfmathparse{\xl+(0.7*\xscale)}		\let\Xdd\pgfmathresult
    \pgfmathparse{\yb+(0.9*\yheight)}		\let\Ydd\pgfmathresult
    \coordinate (D') at (\Xdd,\Ydd);
	\pgfmathparse{\xc}		\let\Xe\pgfmathresult
    \pgfmathparse{\yc}		\let\Ye\pgfmathresult
    \coordinate (E) at (\Xe,\Ye);

\draw[line width=1pt, ->] (A) .. controls (B') and (D') .. (C);
\draw (\xc, \yc) node[shape=circle, fill=white] {};
\draw[line width=1pt, ->] (B) .. controls (A') and (C') .. (D);
}
\newcommand{\SWNW}[4]{
\tikzstyle{cross line}=[preaction={draw=white, -,line width=10pt}];
\pgfmathsetmacro{\xl}{{#1}}
\pgfmathsetmacro{\yb}{{#2}}
\pgfmathsetmacro{\yheight}{{#3}}
\pgfmathsetmacro{\xscale}{{#4}}
\pgfmathsetmacro{\yt}{\yb+\yheight};
\pgfmathsetmacro{\xr}{\xscale+\xl};
\pgfmathsetmacro{\yc}{\yb+\yheight/2};
\pgfmathsetmacro{\xc}{\xscale/2+\xl};

	\pgfmathparse{\xl}		\let\Xa\pgfmathresult
    \pgfmathparse{\yt}		\let\Ya\pgfmathresult
    \coordinate (A) at (\Xa,\Ya);
	\pgfmathparse{\xl+(0.3*\xscale)}		\let\Xaa\pgfmathresult
    \pgfmathparse{\yb+(0.9*\yheight)}		\let\Yaa\pgfmathresult
    \coordinate (A') at (\Xaa,\Yaa);
	\pgfmathparse{\xl}		\let\Xb\pgfmathresult
    \pgfmathparse{\yb}		\let\Yb\pgfmathresult
    \coordinate (B) at (\Xb,\Yb);
	\pgfmathparse{\xl+(0.3*\xscale)}		\let\Xbb\pgfmathresult
    \pgfmathparse{\yb+(0.1*\yheight)}		\let\Ybb\pgfmathresult
    \coordinate (B') at (\Xbb,\Ybb);
	\pgfmathparse{\xr}		\let\Xc\pgfmathresult
    \pgfmathparse{\yb}		\let\Yc\pgfmathresult
    \coordinate (C) at (\Xc,\Yc);
	\pgfmathparse{\xl+(0.7*\xscale)}		\let\Xcc\pgfmathresult
    \pgfmathparse{\yb+(0.1*\yheight)}		\let\Ycc\pgfmathresult
    \coordinate (C') at (\Xcc,\Ycc);
	\pgfmathparse{\xr}		\let\Xd\pgfmathresult
    \pgfmathparse{\yt}		\let\Yd\pgfmathresult
    \coordinate (D) at (\Xd,\Yd);
	\pgfmathparse{\xl+(0.7*\xscale)}		\let\Xdd\pgfmathresult
    \pgfmathparse{\yb+(0.9*\yheight)}		\let\Ydd\pgfmathresult
    \coordinate (D') at (\Xdd,\Ydd);
	\pgfmathparse{\xc}		\let\Xe\pgfmathresult
    \pgfmathparse{\yc}		\let\Ye\pgfmathresult
    \coordinate (E) at (\Xe,\Ye);

\draw[line width=1pt, <-] (A) .. controls (B') and (D') .. (C);
\draw (\xc, \yc) node[shape=circle, fill=white] {};
\draw[line width=1pt, <-] (B) .. controls (A') and (C') .. (D);
}

\newcommand{\SENE}[4]{
\tikzstyle{cross line}=[preaction={draw=white, -,line width=10pt}];
\pgfmathsetmacro{\xl}{{#1}}
\pgfmathsetmacro{\yb}{{#2}}
\pgfmathsetmacro{\yheight}{{#3}}
\pgfmathsetmacro{\xscale}{{#4}}
\pgfmathsetmacro{\yt}{\yb+\yheight};
\pgfmathsetmacro{\xr}{\xscale+\xl};
\pgfmathsetmacro{\yc}{\yb+\yheight/2};
\pgfmathsetmacro{\xc}{\xscale/2+\xl};

	\pgfmathparse{\xl}		\let\Xa\pgfmathresult
    \pgfmathparse{\yt}		\let\Ya\pgfmathresult
    \coordinate (A) at (\Xa,\Ya);
	\pgfmathparse{\xl+(0.3*\xscale)}		\let\Xaa\pgfmathresult
    \pgfmathparse{\yb+(0.9*\yheight)}		\let\Yaa\pgfmathresult
    \coordinate (A') at (\Xaa,\Yaa);
	\pgfmathparse{\xl}		\let\Xb\pgfmathresult
    \pgfmathparse{\yb}		\let\Yb\pgfmathresult
    \coordinate (B) at (\Xb,\Yb);
	\pgfmathparse{\xl+(0.3*\xscale)}		\let\Xbb\pgfmathresult
    \pgfmathparse{\yb+(0.1*\yheight)}		\let\Ybb\pgfmathresult
    \coordinate (B') at (\Xbb,\Ybb);
	\pgfmathparse{\xr}		\let\Xc\pgfmathresult
    \pgfmathparse{\yb}		\let\Yc\pgfmathresult
    \coordinate (C) at (\Xc,\Yc);
	\pgfmathparse{\xl+(0.7*\xscale)}		\let\Xcc\pgfmathresult
    \pgfmathparse{\yb+(0.1*\yheight)}		\let\Ycc\pgfmathresult
    \coordinate (C') at (\Xcc,\Ycc);
	\pgfmathparse{\xr}		\let\Xd\pgfmathresult
    \pgfmathparse{\yt}		\let\Yd\pgfmathresult
    \coordinate (D) at (\Xd,\Yd);
	\pgfmathparse{\xl+(0.7*\xscale)}		\let\Xdd\pgfmathresult
    \pgfmathparse{\yb+(0.9*\yheight)}		\let\Ydd\pgfmathresult
    \coordinate (D') at (\Xdd,\Ydd);
	\pgfmathparse{\xc}		\let\Xe\pgfmathresult
    \pgfmathparse{\yc}		\let\Ye\pgfmathresult
    \coordinate (E) at (\Xe,\Ye);

\draw[line width=1pt, ->] (B) .. controls (A') and (C') .. (D);
\draw (\xc, \yc) node[shape=circle, fill=white] {};
\draw[line width=1pt, ->] (A) .. controls (B') and (D') .. (C);
}
\newcommand{\NWSW}[4]{
\tikzstyle{cross line}=[preaction={draw=white, -,line width=10pt}];
\pgfmathsetmacro{\xl}{{#1}}
\pgfmathsetmacro{\yb}{{#2}}
\pgfmathsetmacro{\yheight}{{#3}}
\pgfmathsetmacro{\xscale}{{#4}}
\pgfmathsetmacro{\yt}{\yb+\yheight};
\pgfmathsetmacro{\xr}{\xscale+\xl};
\pgfmathsetmacro{\yc}{\yb+\yheight/2};
\pgfmathsetmacro{\xc}{\xscale/2+\xl};

	\pgfmathparse{\xl}		\let\Xa\pgfmathresult
    \pgfmathparse{\yt}		\let\Ya\pgfmathresult
    \coordinate (A) at (\Xa,\Ya);
	\pgfmathparse{\xl+(0.3*\xscale)}		\let\Xaa\pgfmathresult
    \pgfmathparse{\yb+(0.9*\yheight)}		\let\Yaa\pgfmathresult
    \coordinate (A') at (\Xaa,\Yaa);
	\pgfmathparse{\xl}		\let\Xb\pgfmathresult
    \pgfmathparse{\yb}		\let\Yb\pgfmathresult
    \coordinate (B) at (\Xb,\Yb);
	\pgfmathparse{\xl+(0.3*\xscale)}		\let\Xbb\pgfmathresult
    \pgfmathparse{\yb+(0.1*\yheight)}		\let\Ybb\pgfmathresult
    \coordinate (B') at (\Xbb,\Ybb);
	\pgfmathparse{\xr}		\let\Xc\pgfmathresult
    \pgfmathparse{\yb}		\let\Yc\pgfmathresult
    \coordinate (C) at (\Xc,\Yc);
	\pgfmathparse{\xl+(0.7*\xscale)}		\let\Xcc\pgfmathresult
    \pgfmathparse{\yb+(0.1*\yheight)}		\let\Ycc\pgfmathresult
    \coordinate (C') at (\Xcc,\Ycc);
	\pgfmathparse{\xr}		\let\Xd\pgfmathresult
    \pgfmathparse{\yt}		\let\Yd\pgfmathresult
    \coordinate (D) at (\Xd,\Yd);
	\pgfmathparse{\xl+(0.7*\xscale)}		\let\Xdd\pgfmathresult
    \pgfmathparse{\yb+(0.9*\yheight)}		\let\Ydd\pgfmathresult
    \coordinate (D') at (\Xdd,\Ydd);
	\pgfmathparse{\xc}		\let\Xe\pgfmathresult
    \pgfmathparse{\yc}		\let\Ye\pgfmathresult
    \coordinate (E) at (\Xe,\Ye);

\draw[line width=1pt, <-] (B) .. controls (A') and (C') .. (D);
\draw (\xc, \yc) node[shape=circle, fill=white] {};
\draw[line width=1pt, <-] (A) .. controls (B') and (D') .. (C);
}

\newcommand{\rcrossdown}[4]{
\tikzstyle{cross line}=[preaction={draw=white, -,line width=10pt}];
\pgfmathsetmacro{\xl}{{#1}}
\pgfmathsetmacro{\yb}{{#2}}
\pgfmathsetmacro{\yheight}{{#3}}
\pgfmathsetmacro{\xscale}{{#4}}
\pgfmathsetmacro{\yt}{\yb+\yheight};
\pgfmathsetmacro{\xr}{\xscale+\xl};
\pgfmathsetmacro{\yc}{\yb+\yheight/2};
\pgfmathsetmacro{\xc}{\xscale/2+\xl};

	\pgfmathparse{\xl}		\let\Xa\pgfmathresult
    \pgfmathparse{\yt}		\let\Ya\pgfmathresult
    \coordinate (A) at (\Xa,\Ya);
	\pgfmathparse{\xl+(0.3*\xscale)}		\let\Xaa\pgfmathresult
    \pgfmathparse{\yb+(0.9*\yheight)}		\let\Yaa\pgfmathresult
    \coordinate (A') at (\Xaa,\Yaa);
	\pgfmathparse{\xl}		\let\Xb\pgfmathresult
    \pgfmathparse{\yb}		\let\Yb\pgfmathresult
    \coordinate (B) at (\Xb,\Yb);
	\pgfmathparse{\xl+(0.3*\xscale)}		\let\Xbb\pgfmathresult
    \pgfmathparse{\yb+(0.1*\yheight)}		\let\Ybb\pgfmathresult
    \coordinate (B') at (\Xbb,\Ybb);
	\pgfmathparse{\xr}		\let\Xc\pgfmathresult
    \pgfmathparse{\yb}		\let\Yc\pgfmathresult
    \coordinate (C) at (\Xc,\Yc);
	\pgfmathparse{\xl+(0.7*\xscale)}		\let\Xcc\pgfmathresult
    \pgfmathparse{\yb+(0.1*\yheight)}		\let\Ycc\pgfmathresult
    \coordinate (C') at (\Xcc,\Ycc);
	\pgfmathparse{\xr}		\let\Xd\pgfmathresult
    \pgfmathparse{\yt}		\let\Yd\pgfmathresult
    \coordinate (D) at (\Xd,\Yd);
	\pgfmathparse{\xl+(0.7*\xscale)}		\let\Xdd\pgfmathresult
    \pgfmathparse{\yb+(0.9*\yheight)}		\let\Ydd\pgfmathresult
    \coordinate (D') at (\Xdd,\Ydd);
	\pgfmathparse{\xc}		\let\Xe\pgfmathresult
    \pgfmathparse{\yc}		\let\Ye\pgfmathresult
    \coordinate (E) at (\Xe,\Ye);

\draw[line width=1pt, ->] (A) .. controls (B') and (D') .. (C);
\draw (\xc, \yc) node[shape=circle, fill=white] {};
\draw[line width=1pt, <-] (B) .. controls (A') and (C') .. (D);
}

\newcommand{\lcross}[4]{
\tikzstyle{cross line}=[preaction={draw=white, -,line width=10pt}];
\pgfmathsetmacro{\xl}{{#1}}
\pgfmathsetmacro{\yb}{{#2}}
\pgfmathsetmacro{\yheight}{{#3}}
\pgfmathsetmacro{\xscale}{{#4}}
\pgfmathsetmacro{\yt}{\yb+\yheight};
\pgfmathsetmacro{\xr}{\xscale+\xl};
\pgfmathsetmacro{\yc}{\yb+\yheight/2};
\pgfmathsetmacro{\xc}{\xscale/2+\xl};

	\pgfmathparse{\xl}		\let\Xa\pgfmathresult
    \pgfmathparse{\yt}		\let\Ya\pgfmathresult
    \coordinate (A) at (\Xa,\Ya);
	\pgfmathparse{\xl+(0.3*\xscale)}		\let\Xaa\pgfmathresult
    \pgfmathparse{\yb+(0.9*\yheight)}		\let\Yaa\pgfmathresult
    \coordinate (A') at (\Xaa,\Yaa);
	\pgfmathparse{\xl}		\let\Xb\pgfmathresult
    \pgfmathparse{\yb}		\let\Yb\pgfmathresult
    \coordinate (B) at (\Xb,\Yb);
	\pgfmathparse{\xl+(0.3*\xscale)}		\let\Xbb\pgfmathresult
    \pgfmathparse{\yb+(0.1*\yheight)}		\let\Ybb\pgfmathresult
    \coordinate (B') at (\Xbb,\Ybb);
	\pgfmathparse{\xr}		\let\Xc\pgfmathresult
    \pgfmathparse{\yb}		\let\Yc\pgfmathresult
    \coordinate (C) at (\Xc,\Yc);
	\pgfmathparse{\xl+(0.7*\xscale)}		\let\Xcc\pgfmathresult
    \pgfmathparse{\yb+(0.1*\yheight)}		\let\Ycc\pgfmathresult
    \coordinate (C') at (\Xcc,\Ycc);
	\pgfmathparse{\xr}		\let\Xd\pgfmathresult
    \pgfmathparse{\yt}		\let\Yd\pgfmathresult
    \coordinate (D) at (\Xd,\Yd);
	\pgfmathparse{\xl+(0.7*\xscale)}		\let\Xdd\pgfmathresult
    \pgfmathparse{\yb+(0.9*\yheight)}		\let\Ydd\pgfmathresult
    \coordinate (D') at (\Xdd,\Ydd);
	\pgfmathparse{\xc}		\let\Xe\pgfmathresult
    \pgfmathparse{\yc}		\let\Ye\pgfmathresult
    \coordinate (E) at (\Xe,\Ye);

\draw[line width=1pt] (B) .. controls (A') and (C') .. (D);
\draw (\xc, \yc) node[shape=circle, fill=white] {};
\draw[line width=1pt] (A) .. controls (B') and (D') .. (C);
}

\newcommand{\lcrossup}[4]{
\tikzstyle{cross line}=[preaction={draw=white, -,line width=10pt}];
\pgfmathsetmacro{\xl}{{#1}}
\pgfmathsetmacro{\yb}{{#2}}
\pgfmathsetmacro{\yheight}{{#3}}
\pgfmathsetmacro{\xscale}{{#4}}
\pgfmathsetmacro{\yt}{\yb+\yheight};
\pgfmathsetmacro{\xr}{\xscale+\xl};
\pgfmathsetmacro{\yc}{\yb+\yheight/2};
\pgfmathsetmacro{\xc}{\xscale/2+\xl};

	\pgfmathparse{\xl}		\let\Xa\pgfmathresult
    \pgfmathparse{\yt}		\let\Ya\pgfmathresult
    \coordinate (A) at (\Xa,\Ya);
	\pgfmathparse{\xl+(0.3*\xscale)}		\let\Xaa\pgfmathresult
    \pgfmathparse{\yb+(0.9*\yheight)}		\let\Yaa\pgfmathresult
    \coordinate (A') at (\Xaa,\Yaa);
	\pgfmathparse{\xl}		\let\Xb\pgfmathresult
    \pgfmathparse{\yb}		\let\Yb\pgfmathresult
    \coordinate (B) at (\Xb,\Yb);
	\pgfmathparse{\xl+(0.3*\xscale)}		\let\Xbb\pgfmathresult
    \pgfmathparse{\yb+(0.1*\yheight)}		\let\Ybb\pgfmathresult
    \coordinate (B') at (\Xbb,\Ybb);
	\pgfmathparse{\xr}		\let\Xc\pgfmathresult
    \pgfmathparse{\yb}		\let\Yc\pgfmathresult
    \coordinate (C) at (\Xc,\Yc);
	\pgfmathparse{\xl+(0.7*\xscale)}		\let\Xcc\pgfmathresult
    \pgfmathparse{\yb+(0.1*\yheight)}		\let\Ycc\pgfmathresult
    \coordinate (C') at (\Xcc,\Ycc);
	\pgfmathparse{\xr}		\let\Xd\pgfmathresult
    \pgfmathparse{\yt}		\let\Yd\pgfmathresult
    \coordinate (D) at (\Xd,\Yd);
	\pgfmathparse{\xl+(0.7*\xscale)}		\let\Xdd\pgfmathresult
    \pgfmathparse{\yb+(0.9*\yheight)}		\let\Ydd\pgfmathresult
    \coordinate (D') at (\Xdd,\Ydd);
	\pgfmathparse{\xc}		\let\Xe\pgfmathresult
    \pgfmathparse{\yc}		\let\Ye\pgfmathresult
    \coordinate (E) at (\Xe,\Ye);

\draw[line width=1pt, ->] (B) .. controls (A') and (C') .. (D);
\draw (\xc, \yc) node[shape=circle, fill=white] {};
\draw[line width=1pt, <-] (A) .. controls (B') and (D') .. (C);
}

\newcommand{\lcrossdown}[4]{
\tikzstyle{cross line}=[preaction={draw=white, -,line width=10pt}];
\pgfmathsetmacro{\xl}{{#1}}
\pgfmathsetmacro{\yb}{{#2}}
\pgfmathsetmacro{\yheight}{{#3}}
\pgfmathsetmacro{\xscale}{{#4}}
\pgfmathsetmacro{\yt}{\yb+\yheight};
\pgfmathsetmacro{\xr}{\xscale+\xl};
\pgfmathsetmacro{\yc}{\yb+\yheight/2};
\pgfmathsetmacro{\xc}{\xscale/2+\xl};

	\pgfmathparse{\xl}		\let\Xa\pgfmathresult
    \pgfmathparse{\yt}		\let\Ya\pgfmathresult
    \coordinate (A) at (\Xa,\Ya);
	\pgfmathparse{\xl+(0.3*\xscale)}		\let\Xaa\pgfmathresult
    \pgfmathparse{\yb+(0.9*\yheight)}		\let\Yaa\pgfmathresult
    \coordinate (A') at (\Xaa,\Yaa);
	\pgfmathparse{\xl}		\let\Xb\pgfmathresult
    \pgfmathparse{\yb}		\let\Yb\pgfmathresult
    \coordinate (B) at (\Xb,\Yb);
	\pgfmathparse{\xl+(0.3*\xscale)}		\let\Xbb\pgfmathresult
    \pgfmathparse{\yb+(0.1*\yheight)}		\let\Ybb\pgfmathresult
    \coordinate (B') at (\Xbb,\Ybb);
	\pgfmathparse{\xr}		\let\Xc\pgfmathresult
    \pgfmathparse{\yb}		\let\Yc\pgfmathresult
    \coordinate (C) at (\Xc,\Yc);
	\pgfmathparse{\xl+(0.7*\xscale)}		\let\Xcc\pgfmathresult
    \pgfmathparse{\yb+(0.1*\yheight)}		\let\Ycc\pgfmathresult
    \coordinate (C') at (\Xcc,\Ycc);
	\pgfmathparse{\xr}		\let\Xd\pgfmathresult
    \pgfmathparse{\yt}		\let\Yd\pgfmathresult
    \coordinate (D) at (\Xd,\Yd);
	\pgfmathparse{\xl+(0.7*\xscale)}		\let\Xdd\pgfmathresult
    \pgfmathparse{\yb+(0.9*\yheight)}		\let\Ydd\pgfmathresult
    \coordinate (D') at (\Xdd,\Ydd);
	\pgfmathparse{\xc}		\let\Xe\pgfmathresult
    \pgfmathparse{\yc}		\let\Ye\pgfmathresult
    \coordinate (E) at (\Xe,\Ye);

\draw[line width=1pt, <-] (B) .. controls (A') and (C') .. (D);
\draw (\xc, \yc) node[shape=circle, fill=white] {};
\draw[line width=1pt, ->] (A) .. controls (B') and (D') .. (C);
}

\title[]{On the center of two-parameter $(v,t)$-quantum groups}
\makeindex

\author[Z. Fan and Z. Xin]{Zhaobing Fan and Ziqi Xin}

\address{College of Mathematical Sciences, Harbin Engineering University, Harbin,  China, 150001}
\email{fanzhaobing@hrbeu.edu.cn  (Fan)}

\address{College of Mathematical Sciences, Harbin Engineering University, Harbin,  China, 150001}
\email{2016@hrbeu.edu.cn (Xin)}

\date{\today}

\begin{document}
\begin{abstract}
This paper mainly considers the centre of two-parameter quantum group $U_{v,t}$ \cite{FL} of finite type via an analogue of the Harish-Chandra homomorphism. Through combining the connection  between one-parameter quantum group case and two-parameter quantum group case, we get the description of centre in the sense of  Harish-Chandra homomorphism.
\end{abstract}
\maketitle

\setcounter{tocdepth}{1}
\tableofcontents

\section{Introduction}
 In the one-parameter quantum group case,  much work has been done on the centre of quantum groups  for the finite-dimensional simple Lie algebras, and also for generalized Kac-Moody algebras.
The approach taken in many of these papers is to define a nondegenerate bilinear form on the quantum group which is invariant under the adjoint action.   The next step involves constructing an algebra homomorphism $\xi$ called the Harish-Chandra homomorphism and proving the map $\xi$ is injective.The main difficulty lies in determining the image of $\xi$. In the two-parameter case for type $A_{n}$ \cite{BKL}, a new phenomenon arises: the $n$ odd and even cases behave differently. Additional central elements arise when $n$ is odd. This phenomenon complicates the description in two-parameter quantum groups.

In \cite{FL}, the first author and Li provided a noval presentation of
the two-parameter quantum algebra $U_{v,t}(\mathfrak g)$ by a geometric approach, where
both parameters $v$ and $t$ have geometric meaning.
Moreover, this presentation unifies the various quantum algebras in literature.
By various specialization, one can obtain one-parameter quantum algebras \cite{Lusztigbook},
two-parameter quantum algebras \cite{BW}, quantum superalgebras \cite{CFYW} and multi-parameter quantum algebras \cite{HPR}.

In \cite{FX}, they defined a skew-Hopf pairing on the deformed quantum algebra
which unifies various quantum algebras in literatures,
such as one-parameter quantum algebras \cite{Lusztigbook},
two-parameter quantum algebras \cite{BW}, quantum superalgebras \cite{CFYW} and multi-parameter quantum algebras \cite{HPR}.
This provided a tool to construct the ad-invariant bilinear form on two-parameter quantum groups.

In this paper, first we show the Harish-Chandra homomorphism on $U_{v,t}(g)$ is injective. We adjust the prove in \cite{BKL} to make it doesn't rely on the weight module representation. Then we construct some central elements based on weight representation theory and a bilinear form which is invariant under the adjoint action. Next we give a necessary and sufficient condition to show if there are some other central elements. Different from other's work, we consider the result from the perspective of grading on $U_{v,t}$. Furthermore, we define $U_{J}$, a set  of subalgebras  of  $U_{v,t}$  and show the image of Harish-Chandra homomorphism on these subalgebras contains the image of centre on the whole algebra. Finally, we show the existence of central elements on some certain degrees to make the consequence more accurate.

This paper is organized as follows. In Section 2, we recall the definition of two-parameter quantum algebra $U_{v,t}$ from \cite{FL}. Besides, we prove some related properties of Verma module of $U_{v,t}$. In Section 3, we construct the nondegenerate bilinear form on $U_{v,t}$ which is invariant under the adjoint action. By using this bilinear form, we construct some central elements using the  weight representation. In Section 4, we define the Harish-Chandra homomorphism $\xi$ on $U_{v,t}$ and show $\xi$ is injective.  In section 5, we analyse  the centre of two-parameter quantum groups via the connection between  one-parameter quantum groups and two parameter quantum groups. By the analysis, we give a necessary and sufficient condition to show if there exist some other central elements. Then we define a set  of subalgebras of  $U_{v,t}$  and show the image of Harish-Chandra homomorphism on these subalgebras contains the image of centre of the whole algebra and construct some central elements to make the consequence more accurate.

\section{The two-parameter quantum algebra $U_{v,t}$}\label{sec:twoparameter}

We briefly review the definition of the two-parameter quantum algebra $U_{v,t}$ in \cite{FL}.

Given a Cartan datum $(I, \cdot)$, let $\Omega=(\Omega_{ij})_{i,j\in I}$ be an integer matrix satisfying that
\begin{itemize}
  \item[(a)] $\Omega_{ii} \in \mathbb{Z}_{>0}$, $\Omega_{ij}\in \mathbb{Z}_{\leq 0}$ for all $i\neq j \in I$;
  \item[(b)] $\frac{\Omega_{ij}+\Omega_{ji}}{\Omega_{ii}}\in \mathbb{Z}_{\leq 0}$ for all $i\neq j \in I$;
  \item[(c)] the greatest common divisor of all $\Omega_{ii}$ is equal to 1.
\end{itemize}

To $\Omega$, we associate the following three bilinear forms on $\mathbb{Z}[I]$.
\begin{eqnarray*}
\langle i, j\rangle &=&
\Omega_{ij}, \quad \hspace{45pt}\forall i, j\in I. \label{eq47}\\
\begin{bmatrix} i,j \end{bmatrix}&=& 2\delta_{ij} \Omega_{ii} -\Omega_{ij}, \quad \forall i, j\in I. \label{eq48}\\
i\cdot j&=&\langle i, j\rangle +\langle j,i\rangle,  \quad \forall i, j\in I. \label{eq49}
\end{eqnarray*}

The matrix $\Omega$ is called of symmetric type if $\Omega_{ii}=1$, $\forall i \in I$. In this case, the associated Cartan datum is of symmetric type. For simplicity, we assume that $\Omega$ is of symmetric type.
\subsection{The free algebra $'\mathfrak{f}$}\label{sec:freealgebra}
For indeterminates $v$ and $t$,
 we set $v_i=v^{i\cdot i/2}$ and $t_i=t^{i\cdot i/2}$.
 Denoted by $v_{\nu}=\prod_i v_i^{\nu_i},\ t_{\nu}=\prod_i t_i^{\nu_i}$
 and ${\rm tr}(\nu) =\sum_{i\in I} \nu_i \in \mathbb{N}$,
 for any $\nu=\sum \nu_i i\in \mathbb{N}[I]$.

Let ${}'\! \mathfrak{f}$ be the free unital associative algebra over $\mathbb{Q}(v,t)$
generated by the symbols $ \theta_i,\ \forall i\in I$.
By setting the degree of the generator $\theta_i$ to be $i$,
the algebra ${}'\! \mathfrak{f}$ becomes an $\mathbb{N}[I]$-graded algebra.
For any $\nu\in \mathbb{N}[I]$, we denote by ${}'\!\mathfrak{f}_{\nu}$
the subspace of all homogenous elements of degree $\nu$.
We have ${}'\! \mathfrak{f}=\oplus_{\nu\in \mathbb{N}[I]}{}'\! \mathfrak{f}_{\nu}$
and denote by $|x|$ the degree of a homogenous element $x\in {}'\! \mathfrak{f}$.

\subsubsection{The tensor product ${}'\!\mathfrak{f}\otimes {}'\!\mathfrak{f}$}
On the tensor product ${}'\!\mathfrak{f}\otimes {}'\!\mathfrak{f}$,
we define an associative $\mathbb{Q}(v,t)$-algebra structure by
\begin{equation*}\label{eq25}
(x_1 \otimes x_2)(y_1 \otimes y_2)
=v^{|y_1|\cdot|x_2|}t^{\langle|y_1|,|x_2|\rangle-\langle |x_2|,|y_1|\rangle}x_1y_1 \otimes x_2y_2,
\end{equation*}
for homogeneous elements $x_1,x_2,y_1$ and $y_2$ in ${}'\!\mathfrak{f}$.
It is associative since the forms $\langle , \rangle$  and $``\cdot "$  are bilinear.

Let $r: {}'\!\mathfrak{f}\rightarrow  {}'\!\mathfrak{f}\otimes {}'\!\mathfrak{f}$
be the $\mathbb{Q}(v,t)$-algebra homomorphism such that
$$r(\theta_i)=\theta_i\otimes 1+1 \otimes \theta_i,\quad {\rm for\ all}\ i\in I.$$

\begin{prop}[]\label{prop:bilinearform}\cite[Proposition 13]{FL}
  There is a unique symmetric bilinear form {\rm (,)} on ${}'\!\mathfrak{f}$ with values in $\mathbb{Q}(v,t)$ such that
  \begin{itemize}
    \item[(a)] $(1,1)=1$;
    \item[(b)] $(\theta_i, \theta_j)=\delta_{ij}\frac{1}{1-v^{-2}_i}$,\quad for all $i, j \in I$;
    \item[(c)] $(x, y'y'')=(r(x), y' \otimes y'')$,\quad for all $x, y', y'' \in {}'\!\mathfrak{f}$;
    \item[(d)] $(x'x'', y)=(x'\otimes x'', r(y))$,\quad for all $x', x'', y \in {}'\!\mathfrak{f}$.
  \end{itemize}
  Here the bilinear form on ${}'\!\mathfrak{f} \otimes {}'\!\mathfrak{f}$ is defined by
  \begin{equation*}\label{eq74}
    (x_1 \otimes x_2, y_1 \otimes y_2)=t^{2[|x_1|,|x_2|]}(x_1, y_1) (x_2, y_2).
  \end{equation*}
\end{prop}

\subsubsection{The maps $r_i$ and ${}_ir$}
For any $i\in I$, let $r_i ({\rm resp.}\ {}_ir): {}'\! \mathfrak{f}\rightarrow {}'\! \mathfrak{f}$
be the unique linear map satisfying the following properties.
\begin{equation*}
  \begin{split}
    r_i(1)=0,\ r_i(\theta_j)=\delta_{ij},\ \forall j\in I\ {\rm and}\
    r_i(xy)=v^{i\cdot |y|}t^{\langle |y|,i\rangle-\langle i, |y|\rangle}r_i(x)y+xr_i(y);\\
    {}_ir(1)=0,\ {}_ir(\theta_j)=\delta_{ij},\ \forall j\in I\ {\rm and}\
    {}_ir(xy)={}_ir(x)y+v^{i\cdot |x|}t^{\langle i,|x|\rangle-\langle |x|,i\rangle}x{}_ir(y).
  \end{split}
\end{equation*}
By an induction on $|x|$, we can show that $r(x)=r_i(x)\otimes \theta_i$ (resp. $r(x)= \theta_i\otimes{}_ir(x)$)
plus other terms.

\subsubsection{Quantum serre relations}

Let $\mathfrak{J}$ be the radical of the bilinear form $(-,-)$.
It is clear that $\mathfrak{J}$ is a two-sided ideal of ${}'\!\mathfrak{f}$.
Denote the quotient algebra of ${}'\!\mathfrak{f}$ by
\[\mathfrak{f}={}'\!\mathfrak{f}/\mathfrak{J}.\]

Recall the quantum integers from \cite{FL}.
For any $n\in \mathbb{N}$, we have
$$[n]_{v, t}=\frac{(vt)^n-(vt^{-1})^{-n}}{vt-(vt^{-1})^{-1}},\qquad
[n]^!_{v, t}=\prod_{k=1}^n[k]_{v,t}.$$
Denote by
 $$\theta_i^{(n)}=\frac{\theta_i^n}{[n]^!_{v_i, t_i}}.$$

\begin{prop}\label{prop15}~\cite[Proposition 14]{FL}
  The generators $\theta_i$ of $\mathfrak{f}$ satisfy the following identities.
  $$
  \sum_{p+p'=1-2\frac{i\cdot j}{i\cdot i}}
  (-1)^pt_i^{-p(p'-2\frac{\langle i,j\rangle}{i\cdot i}+
  2\frac{\langle j,i\rangle}{i\cdot i})}\theta_i^{(p)}\theta_j\theta_i^{(p')}=0,
  \quad \forall i\not =j\in I.
  $$
\end{prop}

\subsection{ The two-parameter quantum algebra $U_{v,t}$}\label{sec:presentation}
By Drinfeld double construction, we get the following presentation of
the entire two-parameter quantum algebra $U_{v,t}$,
generated by symbols $E_i, F_i, K_i^{\pm 1}, K_i'^{\pm 1}, \forall i\in I$,
and subjects to the following relations.
\allowdisplaybreaks
\begin{eqnarray*}
  (R1)& &K_i^{\pm 1}K_i^{\mp 1}=K'^{\pm 1}_iK'^{\mp 1}_i=1.\\
  (R2)& &K_iE_jK^{-1}_i=v^{i\cdot j}t^{\langle j,i\rangle-\langle i,j\rangle} E_j,\ \ K'_iE_jK'^{-1}_i=v^{-i\cdot j}t^{\langle j,i\rangle-\langle i,j\rangle}E_j,\\
      & &K'_iF_jK'^{-1}_i=v^{i\cdot j}t^{\langle i,j\rangle-\langle j,i\rangle} F_j,\ \ K_iF_jK^{-1}_i=v^{-i\cdot j}t^{\langle i,j\rangle-\langle j,i \rangle}F_j.\\
  (R3)& &E_iF_j-F_j E_i=\delta_{ij}\frac{K_i-K'_i}{v_i-v^{-1}_i}.\\
  (R4)& &\sum_{p+p'=1-2\frac{i\cdot j}{i\cdot i}}(-1)^pt_i^{-p(p'-2\frac{\langle i,j\rangle}{i\cdot i}+2\frac{\langle j,i\rangle}{i\cdot i})}
         E_i^{(p)}E_j E_i^{(p')}=0, \quad {\rm if}\ i\not =j,\\
  & &\sum_{p+p'=1-2\frac{i\cdot j}{i\cdot i}}(-1)^pt_i^{-p(p'-2\frac{\langle i,j\rangle}{i\cdot i}+2\frac{\langle j,i\rangle}{i\cdot i})}
         F_i^{(p')}F_j F_i^{(p)}=0, \quad {\rm if}\ i\not =j.\\
\end{eqnarray*}

 The algebra $U_{v,t}$ has a Hopf algebra structure with the comultiplication $\Delta$, the counit $\varepsilon$ and the antipode $S$ given as follows.
  $$\begin{array}{llll}
   &\Delta(K_i^{\pm 1})=K_i^{\pm 1} \otimes K_i^{\pm 1},&\Delta(K'^{\pm 1}_i)=K'^{\pm 1}_i \otimes K'^{\pm 1}_i, &\vspace{4pt}\\
   &\Delta(E_i)=E_i\otimes 1+K_i\otimes E_i,& \Delta(F_i)=1 \otimes F_i+F_i\otimes K'_i, & \vspace{4pt}\\
  &\varepsilon(K_i^{\pm 1})=\varepsilon(K'^{\pm 1}_i)=1,& \varepsilon(E_i)=\varepsilon(F_i)=0,& S(K_i^{\pm 1})=K_i^{\mp 1},\vspace{4pt}\\
  & S(K'^{\pm 1}_i)=K'^{\mp 1}_i,&
  S(E_i)=-K_i^{-1}E_i,&S(F_i)=-F_iK'^{-1}_i.
  \end{array}$$

Let $U_{v,t}^+$(resp. $U_{v,t}^-$) be a subalgebra of $U_{v,t}$ generated by $E_i$(resp. $F_i$).

\subsection{The module of $U_{v,t}$}
A $U_{v,t}$-module M is called  weight module if it admits a decomposition $M=\bigoplus_{\lambda \in \mathbb{N}[I]}M_{\lambda}$ of vector spaces such that
$$M_{\lambda}=\{m \in M| K_{i}m=v^{i\cdot\lambda}c_{i.\lambda}m, K'_{i}m=v^{-i\cdot\lambda}c_{i,\lambda}m, \forall i \in I  \},   $$
where
$$c_{i,\lambda}=t^{\langle\lambda,i\rangle-\langle i,\lambda\rangle}.$$

 let $\Lambda$ and Q be the weight lattice and the root lattice, respectively. And Let $\Lambda^{+}$ denote the dominant weights.

  Fix $\lambda \in \Lambda$ and assume that $ J_{v,t }(\lambda) $ is the left ideal of $U_{v,t}(g)$ generated by $E_{i}$, $K_{i}-v^{i\cdot\lambda}t^{\langle\lambda,i\rangle-\langle i,\lambda\rangle}$, and $K_{i}^{'}-v^{-i\cdot\lambda}t^{\langle\lambda,i\rangle-\langle i,\lambda\rangle}$ for all $i \in I$.  Denote by $M_{v,t}(\lambda)=U_{v,t}(g)/J_{v,t}(\lambda)$ the Verma module of two-parameter quantum group. This is a $U_{v,t}$-module by left multiplication and the highest weight vector $v_{\lambda}=1+J_{v,t}(\lambda)$.

 By analogue with the one-parameter quantum group case, $M_{v,t}(g)$ has the following proposition.
\begin{prop} $M_{v,t}(\lambda)$ satisfies:
\begin{itemize}
   \item[(1)] As a $U_{v,t}$-module, $M_{v,t}(\lambda)$ is free of rank 1, generated by the highest weight vector $v_{\lambda}=1+J_{v,t}(\lambda)$ ;
   \item[(2)] Every highest weight  $U_{v,t}$-module with highest weight $\lambda$ is a homomorphic image of $M_{v,t}(g)$;
   \item[(3)] The Verma module $M_{v,t}(g)$ has a unique maximal submodule.
\end{itemize}
\end{prop}
\begin{proof}
The proof is almost the same with the one parameter quantum group case.
\end{proof}

\begin{thm}
For $\lambda \in \Lambda^{+}$, $i \in I$, set $n=\lambda\cdot i+1$, then there is a  homomorphism of $U_{v,t}$-module
$$f : M_{v,t}(\lambda-ni) \mapsto M_{v,t}(\lambda)$$
 with $f(v_{\lambda-ni})=F^{n}_{i}v_{\lambda}$.
\end{thm}
\begin{proof}
 From the universal property of $M_{v,t}(g)$, it is enough to show that $E_{j}F^{n}_{i}v_{\lambda}=0$ for all $j \in I$. If $j\neq i$, $E_{j}F^{n}_{i}v_{\lambda}=0$ is obvious. We only have to check the case when $j=i$. The communication relation of $E_{i}$ and $F^{n}_{i}$ is as follows
 $$E_{i}F^{n}_{i}-F^{n}_{i}E_{i}=\frac{\sum^{n-1}_{k=0}(F^{k}_{i}K_{i}F^{n-1-k}_{i}-F^{k}_{i}K'_{i}F^{n-1-k}_{i})}{v_{i}-v^{-1}_{i}}$$
 $$=\frac{\sum^{n-1}_{k=0}(v^{-(n-1-k)i\cdot i}F^{n-1}_{i}K_{i}-v^{(n-1-k)i \cdot i}F^{n-1}_{i}K'_{i})}{v_{i}-v^{-1}_{i}}.$$
Since $E_{i}v_{\lambda}=0$, we have
 $$E_{i}F^{n}_{i}v_{\lambda}=(v_{i}-v^{-1}_{i})^{-1}F^{n-1}_{i}(\frac{1-v^{-2n}_{i}}{1-v_{i}^{-2}}K_{i}-\frac{1-v^{2n}_{i}}{1-v_{i}^{2}}K'_{i})v_{\lambda}.$$
 So $E_{i}v_{\lambda}=0$ is equivalent to say
 $$(\frac{1-v^{-2n}_{i}}{1-v_{i}^{-2}}K_{i}-\frac{1-v^{2n}_{i}}{1-v_{i}^{2}}K'_{i})v_{\lambda}=0.$$
 From the above equation, we get $n=\lambda\cdot i+1$.
 \end{proof}

\section{The ad-invariant bilinear form }

 \begin{prop} \cite{FX}
 There exists a unique bilinear form $(\cdot,\cdot): U^{\leq}_{v,t}\times U^{\geq}_{v,t}\rightarrow \mathbb{Q}(v,t)$ such that for all $x, x' \in U^{\leq}_{v,t}$, $y,y' \in U^{\geq}_{v,t}$, $\mu,\nu \in \mathbb{Z}[I]$ and $i,j \in I$, we have
 \begin{equation*}
\begin{aligned}
(x,yy')&=(\Delta(x),y\otimes y'),\ (xx',y)=(x'\otimes x,\Delta(y)),\ (F_{i},E_{i})=\delta_{ij}(v_{i}^{-1}-v_{i})^{-1},\\
(K'_{\mu},K_{\nu})&=v^{\mu \cdot \nu }t^{\langle\nu,\mu\rangle-\langle\mu,\nu\rangle},\ (K'_{\mu},E_{i})=(F_{i},K_{\mu})=0,\\
\end{aligned}
\end{equation*}
where $K_{\mu}=\prod_{i\in I}K^{\mu_{i}}_{i}$, $K'_{\mu}\prod_{i\in I}K^{\mu'_{i}}_{i}$ for each $\mu=\Sigma_{i \in I}\mu_{i}\alpha_{i} \in Q$.
 \end{prop}

The above bilinear form is the skew-Hopf pair of $U_{v,t}(g)$. Since we will frequently use the commutation relation between the positive part and negative part, we need the following lemma.
\begin{lem}   For each $i \in I$, there exist two linear maps $p_{i}$ and $p'_{i}$ on $U^{+}_{v,t}$, such that for any $x \in U^{+}_{v,t}$
$$xF_{i}-F_{i}x=\frac{p_{i}(x)K_{i}-K'_{i}p'_{i}(x)}{v_{i}-v^{-1}_{i}}$$
and
 \begin{equation*}
\begin{aligned}
\Delta(x)&=x\otimes1+\sum_{i\in I}p_{i}(x)K_{i}\otimes E_{i}+M_{1}\\
 &=K_{\mu}\otimes x+\sum_{i\in I}E_{i}K_{\mu-i}\otimes p'_{i}(x)+M_{2},\\
\end{aligned}
\end{equation*}
where $M_{1}$ and $M_{2}$ denote the sums of the other terms.
\end{lem}
\begin{proof}

Consider the definition of co-product, we can get
$$ \Delta(U^{+}_{\xi})=\bigoplus_{0\leq\nu\leq\xi}U^{+}_{\xi-\nu}K_{\nu}\otimes U^{+}_{\nu},  $$
where $\leq$ is the usual partial order on $Q^{+}$ : $\nu\leq\mu $ if $\nu-\mu \in Q^{+}$. Thus, for each $i \in I$, we can define two linear maps $p_{i}$ and $p'_{i}$ on $U^{+}_{v,t}$, satisfying the following equation

\begin{equation*}
\begin{aligned}
\Delta(x)&=x\otimes1+\sum_{i\in I}p_{i}(x)K_{i}\otimes E_{i}+M_{1}\\
 &=K_{\mu}\otimes x+\sum_{i\in I}E_{i}K_{\mu-i}\otimes p'_{i}(x)+M_{2},\\
\end{aligned}
\end{equation*}
where $M_{1}$ and $M_{2}$ denote the sums of the other terms.

On the other hand, for $x \in U^{+}_{v,t}$ and $i \in I$, we have
$$xF_{i}-F_{i}x=\frac{b_{i}(x)K_{i}-K'_{i}b'_{i}(x)}{v_{i}-v^{-1}_{i}},$$
where $b_{i}$ and $b'_{i}$ are linear maps on $U^{+}_{v,t}$.

We have to show the map $b_{i}$  is equal to  $p_{i}$ and  the map $b'_{i}$ is equal to $p'_{i}$.

Assume $x_{1}$ and $x_{2}$ are two homogeneous elements in $U^{+}_{v,t}$, we have
$$x_{1}x_{2}F_{i}-F_{i}x_{1}x_{2}=$$
$$\frac{(x_{1}b_{i}(x_{2})+v^{i\cdot|x_{2}|}t^{\langle i,|x_{2}|\rangle-\langle|x_{2}|,i \rangle}b_{i}(x_{1})x_{2})K_{i}-K'_{i}(v^{i\cdot|x_{1}|}t^{\langle|x_{1}|,i\rangle-\langle i,|x_{1}|\rangle}x_{1}b'_{i}(x_{2})+b'_{i}(x_{1})x_{2})}{v_{i}-v^{-1}_{i}}.$$

From the above equation, we get the following relations:
$$b_{i}(x_{1}x_{2})=x_{1}b_{i}(x_{2})+v^{i\cdot|x_{2}|}t^{\langle i,|x_{2}|\rangle-\langle|x_{2}|,i \rangle}b_{i}(x_{1})x_{2},$$
$$b'_{i}(x_{1}x_{2})=v^{i\cdot|x_{1}|}t^{\langle|x_{1}|,i\rangle-\langle i,|x_{1}|\rangle}x_{1}b'_{i}(x_{2})+b'_{i}(x_{1})x_{2}.$$

On the other hand,
$$\Delta(x_{1}x_{2})=(( v^{i\cdot|x_{2}|}t^{\langle i,|x_{2}|\rangle-\langle|x_{2}|,i \rangle}p_{i}(x_{1})x_{2}+x_{1}p_{i}(x_{2}))   K_{i} \otimes E_{i})\oplus M,$$
where $M$ denotes the sum of the other terms.

So we have the following relation:
$$p_{i}(x_{1}x_{2})= x_{1}p_{i}(x_{2})+v^{i\cdot|x_{2}|}t^{\langle i,|x_{2}|\rangle-\langle|x_{2}|,i \rangle}p_{i}(x_{1})x_{2}. $$

Comparing $b_{i}(x_{1}x_{2})$ with $p_{i}(x_{1}x_{2})$, we conclude $b_{i}(x)=p_{i}(x)$ for all $x \in U^{+}_{v,t}$.

Similarly,
$$\Delta(x_{1}x_{2})=( E_{i}K_{|x_{1}|+|x_{2}|-i}  \otimes v^{i\cdot|x_{1}|}t^{\langle|x_{1}|,i\rangle-\langle i,|x_{1}|\rangle}x_{1}p_{i}'(x_{2})      +         p'_{i}(x_{1})x_{2}    ) \oplus M',$$
where $M'$ denotes the sum of the other terms.

So we have the following relation:
$$p'_{i}(x_{1}x_{2})=v^{i\cdot|x_{1}|}t^{\langle|x_{1}|,i\rangle-\langle i,|x_{1}|\rangle}x_{1}p_{i}'(x_{2})      +         p'_{i}(x_{1})x_{2}.$$

Comparing $b'_{i}(x_{1}x_{2})$ with $p'_{i}(x_{1}x_{2})$, we conclude $b'_{i}(x)=p'_{i}(x)$ for all $x \in U^{+}_{v,t}$.

After the above discussion, we have the map $b_{i}$  is equal to  $p_{i}$ and  the map $b'_{i}$ is equal to $p'_{i}$.
\end{proof}

Parallel to the positive part, there are also two linear maps defined on the negative part.    For any $i \in I$, by analogue with $U^{+}_{v,t}$, denote by $a_{i}$ and $a'_{i}$  the linear maps on $U^{-}_{v,t}$, such that

$$x^{-}E_{i}-E_{i}x^{-}=\frac{a_{i}(x)K'_{i}-K_{i}a'_{i}(x)}{v_{i}-v^{-1}_{i}}$$
for any $x^{-} \in U^{-}_{v,t}$.

Besides, the maps $a_{i}$ and $a'_{i}$ also satisfy
\begin{equation*}
\begin{aligned}
\Delta(x^{-})&=1\otimes x^{-}+\sum_{i\in I}F_{i}\otimes a_{i}(x^{-})K'_{i} +M'_{1}\\
 &=x^{-}\otimes K'_{\mu}+\sum_{i\in I}a'_{i}(x^{-})\otimes F_{i}K'_{\mu-i}+M'_{2}.\\
\end{aligned}
\end{equation*}
where $M'_{1}$ and $M'_{2}$ denote the sums of the other terms.

It is well-known that $U_{v,t}$ can act on itself by adjoint representation. Recall the adjoint representation as  follows,

$$ ad(u)m=\sum_{(u)}u_{(1)}mS(u_{(2)})$$
where $u,m \in U_{v,t}$ and $\Delta(u)=\Sigma_{(u)}u_{(1)} \bigotimes u_{(2)}$.

It is evident from the triangular decomposition that $U_{v,t}$ has the decomposition
$$\bigoplus_{\mu,\upsilon\in Q^{+}}(U_{-\nu}^{-}K'_{-\nu})\otimes(U^{0})\otimes(U_{\mu}^{+})\rightarrow U_{v,t}.$$

Recall the skew-Hopf pair $(,)$ of $U_{v,t}$ and the above decomposition allows us to define the following bilinear form.

\begin{dfn} Set $\rho$ the half sum of positive roots. Define  $$\langle yK'_{-\nu}K_{\eta}'K_{\phi}x,y_{1}K'_{-\nu_{1}}K_{\eta_{1}}'K_{\phi_{1}}x_{1} \rangle=(y ,x_{1})(y_{1} ,x)(K_{\eta}' ,K_{\phi_{1}})( K_{\eta_{1}}',K_{\phi})v^{2\rho\cdot\nu}$$
for all $x \in U_{\mu}^{+}$, $x_{1} \in U_{\mu_{1}}^{+}$, $y \in U_{\nu}^{-}$, $y_{1} \in U_{\nu_{1}}^{-}$, $\nu,\nu_{1},\mu,\mu_{1} \in Q^{+}$ and all $\eta,\eta_{1},\phi,\phi_{1}\in Q$. Extend it to the whole algebra $U_{v,t}$, we get the bilinear form $\langle\cdot|\cdot\rangle$: $U_{v,t}\times U_{v,t}\rightarrow \mathbb{Q}(v,t).$
\end{dfn}

Form the above definition and the non-degeneracy of skew-Hopf pair, we can conclude the following lemma.

\begin{lem} Let $\mu,\mu_{1},\nu,\nu_{1} \in Q^{+}$. Then
$$\langle U^{-}_{-\nu}U^{0}U^{+}_{\mu} |     U^{-}_{-\nu_{1}}U^{0}U^{+}_{\mu_{1}}    \rangle=0$$
unless $\nu=\mu_{1}$ and $\mu=\nu_{1}$.
\end{lem}

\begin{prop} The bilinear form is ad-invariant on $U_{v,t}$. That is, we have
$$\langle ad(u)u_{1} | u_{2}\rangle=\langle u_{1}| ad(S(u))u_{2}\rangle$$
for all $u,u_{1},u_{2} \in U_{v,t}$.
\end{prop}

\begin{proof}  We only have to prove that the bilinear form is ad-invariant when $u$ is one of the generators  $E_{i}$,$F_{i}$,$K_{i}$,$K'_{i}$ for all $i \in I$. And for convenience, assume $u_{1}$ and $u_{2}$ have the following form,

\begin{equation*}
\begin{aligned}
&u_{1}=yK_{\nu}^{-1}K_{\eta}'K_{\phi}x\\
 &u_{2}=y_{1}K_{\nu_{1}}^{-1}K'_{\eta_{1}}K_{\phi_{1}}x_{1}.
\end{aligned}
\end{equation*}
where $|y|=-\nu$, $|y_{1}|=-\nu_{1}$, $|x|=\mu$ and $|x_{1}|=\mu_{1}$.

When $u=K_{i}$, we have

$$ad(K_{i})u_{1}=K_{i}u_{1}K^{-1}_{i}=v^{i \cdot (\mu-\nu)}t^{\langle \nu-\mu ,i \rangle-\langle i,\nu-\mu \rangle}u_{1}.$$

So we get
$$\langle ad(K_{i})u_{1}|u_{2} \rangle=   v^{i \cdot (\mu-\nu)}t^{\langle \nu-\mu ,i \rangle-\langle i,\nu-\mu \rangle}\langle u_{1}|u_{2}\rangle.$$

Besides, we have
$$ ad(S(K_{i}))u_{2}=K^{-1}_{i}u_{2}K_{i}=v^{i \cdot (-\mu_{1}+\nu_{1})}t^{\langle -\nu_{1}+\mu_{1} ,i \rangle-\langle i,-\nu_{1}+\mu_{1} \rangle}u_{2},$$
so we get
$$\langle u_{1}|ad(S(K_{i}))u_{2}\rangle=v^{i \cdot (-\mu_{1}+\nu_{1})}t^{\langle -\nu_{1}+\mu_{1} ,i \rangle-\langle i,-\nu_{1}+\mu_{1} \rangle}\langle u_{1}|u_{2}\rangle.$$

If $\langle u_{1}|u_{2}\rangle=0$, it is obvious to get the conclusion.

If $\langle u_{1}|u_{2}\rangle \neq 0$, we have $\nu-\mu=\mu_{1}-\nu_{1}$ from the \textrm{lemma 3.4}. So we can conclude $\langle ad(K_{i})u_{1}|u_{2}\rangle=\langle u_{1}|ad(S(K_{i}))u_{2}\rangle$.

When $u=F_{i}$, we have
$$ad(F_{i})u_{1}=-u_{1}F_{i}K'^{-1}_{i}+F_{i}u_{1}K'^{-1}_{i}.$$

From \textrm{Lemma 3.2}, we get

$$xF_{i}-F_{i}x=\frac{p_{i}(x)K_{i}-K'_{i}p'_{i}(x)}{v_{i}-v^{-1}_{i}}$$
for any $x \in U^{+}_{v,t}$. Then

\begin{equation*}
\begin{aligned}
ad(F_{i})u_{1}&=u_{1}S(F_{i})+F_{i}u_{1}S(K'_{i})=-u_{1}F_{i}K'^{-1}_{i}+F_{i}u_{1}K'^{-1}_{i}\\
 &=-yK'^{-1}_{\nu}K'_{\eta}K_{\phi}xF_{i}K'^{-1}_{i}+F_{i}yK'^{-1}_{\nu}K'_{\eta}K_{\phi}xK'^{-1}_{i}\\
 &=-yK'^{-1}_{\nu}K'_{\eta}K_{\phi}(F_{i}x+(v_{i}-v^{-1}_{i})^{-1}(p_{i}(x)K_{i}-K'_{i}p'_{i}(x)))K'^{-1}_{i}+F_{i}yK'^{-1}_{\nu}K'_{\eta}K_{\phi}xK'^{-1}_{i}\\
 &=-yK'^{-1}_{\nu}K'_{\eta}K_{\phi}F_{i}xK'^{-1}_{i}+(v_{i}-v^{-1}_{i})^{-1}(yK'^{-1}_{\nu}K'_{\eta}K_{\phi}K'_{i}p'_{i}(x)K'^{-1}_{i})\\
 &  \ \ \ -(v_{i}-v^{-1}_{i})^{-1}(yK'^{-1}_{\nu}K'_{\eta}K_{\phi}p_{i}(x)K_{i}K'^{-1}_{i})+F_{i}yK'^{-1}_{\nu}K'_{\eta}K_{\phi}xK'^{-1}_{i}\\
 &=-v^{(-i)\cdot(\mu+\phi+\nu-\eta)}t^{\langle i,\mu+\phi+\eta-\nu\rangle-\langle\mu+\phi+\eta-\nu,i\rangle}yF_{i}K'^{-1}_{\nu+i}K'_{\eta}K_{\phi}x\\
   &\ \ \  +(v_{i}-v^{-1}_{i})^{-1}v^{(-i)\cdot(\mu-i)}t^{\langle i,\mu-i\rangle-\langle\mu-i,i\rangle}yK'^{-1}_{\nu}K'_{\eta}K_{\phi}p'_{i}(x)\\
   &\ \ \ -(v_{i}-v^{-1}_{i})^{-1}v^{(-2i)\cdot(\mu-i)}yK'^{-1}_{\nu}K'_{\eta-i}K_{\phi+i}p_{i}(x)\\
   &\ \ \ +v^{-i\cdot\mu}t^{\langle i,\mu\rangle-\langle\mu,i\rangle}F_{i}yK'^{-1}_{\nu+i}K'_{\eta}K_{\phi}x
\end{aligned}
\end{equation*}

Now
$$ad(S(F_{i}))u_{2}=ad(-F_{i}K'^{-1}_{i})u_{2}=-v^{i \cdot (\mu_{1}-\nu_{1})}t^{\langle i,\nu_{1}-\mu_{1}\rangle-\langle \nu_{1}-\mu_{1},i\rangle}ad(F_{i})u_{2}.$$

Similarly, repeat the previous calculation with $u_{1}$ replaced by $u_{2}$, we have

\begin{equation*}
\begin{aligned}
ad(S(F_{i}))u_{2}=&v^{-i\cdot(\phi_{1}+2\nu_{1}-\eta_{1})}t^{\langle i,\phi_{1}+\eta_{1}\rangle-\langle\phi_{1}+\eta_{1},i\rangle}y_{1}F_{i}K'^{-1}_{\nu_{1}+i}K'_{\eta_{1}}K_{\phi_{1}}x_{1}\\
&-(v_{i}-v^{-1}_{i})^{-1}v^{-i\cdot(\nu_{1}-i)}t^{\langle i,\nu_{1}-i\rangle-\langle\nu_{1}-i,i\rangle}y_{1}K'^{-1}_{\nu_{1}}K'_{\eta_{1}}K_{\phi_{1}}p'_{i}(x_{1})\\
&+(v_{i}-v^{-1}_{i})^{-1}v^{i\cdot(2i-\mu_{1}-\nu_{1})}t^{\langle i,\nu_{1}-\mu_{1}\rangle-\langle \nu_{1}-\mu_{1},i\rangle}y_{1}K'^{-1}_{\nu_{1}}K'_{\eta_{1}-i}K_{\phi_{1}+i}p_{i}(x_{1})\\
&-v^{-i\cdot\nu_{1}}t^{\langle i,\nu_{1}\rangle-\langle\nu_{1},i\rangle}F_{i}y_{1}K'^{-1}_{\nu_{1}+i}K'_{\eta_{1}}K_{\phi_{1}}x_{1}
\end{aligned}
\end{equation*}

Note that $\langle U^{-}_{-\nu}U^{0}U^{+}_{\mu}| U^{-}_{-\nu_{1}}U^{0}U^{+}_{\mu_{1}}\rangle$=0 unless $\mu_{1}=\nu$ and $\mu=\nu_{1}$. Then $\langle ad (F_{i})u_{1}|u_{2}\rangle$ and $\langle u_{1}|ad(S(F_{i}))u_{2}\rangle$ are non-zero only when  either (a) $\nu+i=\mu_{1}$ and $\mu=\nu_{1}$ , or  (b) $\nu=\mu_{1}$ and $\nu_{1}=\mu-i$. Here, we only prove $(a)$ since $(b)$ can be proved in a similar way.

When $\nu+i=\mu_{1}$ and $\mu=\nu_{1}$. From the definition of $\langle \cdot | \cdot \rangle$,  we have

\begin{equation*}
\begin{aligned}
\langle ad(F_{i})u_{1}|u_{2}\rangle=\langle& -v^{(-i)\cdot(\mu+\phi+\nu-\eta)}t^{\langle i,\mu+\phi+\eta-\nu\rangle-\langle\mu+\phi+\eta-\nu,i\rangle}yF_{i}K'^{-1}_{\nu+i}K'_{\eta}K_{\phi}x|y_{1}K'^{-1}_{\nu_{1}}K'_{\eta_{1}}K_{\phi_{1}}x_{1}\rangle\\
&+\langle v^{-i\cdot\mu}t^{\langle i,\mu\rangle-\langle\mu,i\rangle}F_{i}yK'^{-1}_{\nu+i}K'_{\eta}K_{\phi}x |y_{1}K'^{-1}_{\nu_{1}}K'_{\eta_{1}}K_{\phi_{1}}x_{1}\rangle\\
=&-v^{(-i)\cdot(\mu+\phi+\nu-\eta)}t^{\langle i,\mu+\phi+\eta-\nu\rangle-\langle\mu+\phi+\eta-\nu,i\rangle}\\
&\times(yF_{i},x_{1})(y_{1},x)(K'_{\eta},K_{\phi_{1}})(K'_{\eta_{1}},K_{\phi})v^{2\rho\cdot(\nu+i)}\\
&+v^{-i\cdot\mu}t^{\langle i,\mu\rangle-\langle\mu,i\rangle}(F_{i}y,x)(y_{1},x)(K'_{\eta},K_{\phi_{1}})(K'_{\eta_{1}},K_{\phi})v^{2\rho\cdot(\nu+i)}\\
=&A\times(x,y_{1})(K_{\phi},K'_{\eta_{1}})(K_{\phi_{1}},K'_{\eta})v^{2\rho\cdot(\nu+i)},\\
\end{aligned}
\end{equation*}

where

$A=-v^{(-i)\cdot(\mu+\phi+\nu-\eta)}t^{\langle i,\mu+\phi+\eta-\nu\rangle-\langle\mu+\phi+\eta-\nu,i\rangle}(yF_{i},x) +v^{-i\cdot\mu}t^{\langle i,\mu\rangle-\langle\mu,i\rangle}(F_{i}y,x).$

Similarly,
\begin{equation*}
\begin{aligned}
  \langle yK'_{-\nu}K_{\eta}'K_{\phi}x|ad(S(F_{i}))u_{2}\rangle=B\times(x,y_{1})(K_{\phi},K'_{\eta_{1}})(K_{\phi_{1}},K'_{\eta})v^{2\rho\cdot\nu},\\
\end{aligned}
\end{equation*}

where

\begin{equation*}
\begin{aligned}
B&=-(v_{i}-v^{-1}_{i})^{-1}v^{(-i)\cdot(\nu_{1}-i)}t^{\langle i,\nu_{1}-i\rangle-\langle\nu_{1}-i,i\rangle}(p'_{i}(x),y)\\
&+(v_{i}-v^{-1}_{i})^{-1}v^{(i)\cdot(2i-\mu_{1}-\nu_{1})}t^{\langle i,\nu_{1}-\mu_{1}\rangle+\langle \nu_{1}-\mu_{1},i\rangle}(p_{i}(x_{1},y)).
\end{aligned}
\end{equation*}

Comparing $\langle ad(F_{i})u_{1}|u_{2}\rangle$ with $\langle u_{1}|ad(S(F_{i}))u_{2}\rangle$, we have $\langle ad(F_{i})u_{1}|u_{2}\rangle=\langle u_{1}|ad(S(F_{i}))u_{2}\rangle$ for all $i \in I$.

When  $u=E_{i}$. This case is similar to the case when $u=F_{i}$.

\end{proof}

In the rest of this paper, the Cartan matrix $(a_{ij})$, where $a_{ij}=i\cdot j,\ \text{for} \ i,j\in I$, is of finite type.

For $\eta,\phi \in Q\times Q$, we define $\chi_{\eta,\phi}$ : $Q\times Q\rightarrow \mathbb{Q}(v,t)$ by
$$\chi_{\eta,\phi}(\eta_{1},\phi_{1})=(K'_{\eta},K_{\phi_{1}})(K'_{\eta_{1}},K_{\phi}),\ \  \text{for all}\ (\eta_{1},\phi_{1})\in Q\times Q.$$
\begin{lem}
If $\chi_{\eta,\phi}=\chi_{\eta',\phi'}$, then $(\eta,\phi)=(\eta',\phi').$
\end{lem}
\begin{proof}
 By the definition, we get that for all $i \in I$

$$\chi_{\eta,\phi}(0,\alpha_{i})=(K'_{\eta},K_{i})=v^{\eta \cdot i }t^{\langle i,\eta\rangle-\langle\eta,i\rangle},$$
$$\chi_{\eta',\phi'}(0,\alpha_{i})=(K'_{\eta'},K_{i})=v^{\eta' \cdot i }t^{\langle i,\eta'\rangle-\langle\eta',i\rangle},$$
$$\chi_{\eta,\phi}(\alpha_{i},0)=(K'_{i},K_{\phi})=v^{i \cdot \phi }t^{\langle \phi,i\rangle-\langle i,\phi\rangle},$$
$$\chi_{\eta',\phi'}(\alpha_{i},0)=(K'_{i},K_{\phi'})=v^{i \cdot \phi' }t^{\langle \phi',i\rangle-\langle i,\phi'\rangle}.$$
Since the Cartan matrix $(a_{ij})$, where $a_{ij}=i\cdot j,\ \text{for} \ i,j\in I$, is non-degenerate, we have $\eta=\eta_{1}$ and $\phi=\phi_{1}$.

\end{proof}
\begin{prop}
The ad-invariant bilinear form $\langle\cdot|\cdot\rangle$ is non-degenerate on $U_{v,t}$.
\end{prop}
\begin{proof}
Since  for $\mu,\mu_{1},\nu,\nu_{1} \in Q^{+}$, $\langle U^{-}_{-\nu}U^{0}U^{+}_{\mu} | U^{-}_{-\nu_{1}}U^{0}U^{+}_{\mu_{1}}\rangle=0$ unless $\nu=\mu_{1}$ and $\mu=\nu_{1}$, it is sufficient to argue that if $u \in U^{-}_{-\mu}U^{0}U^{+}_{\nu}$ and $\langle u|v\rangle=0$ for all $v \in U^{-}_{-\mu}U^{0}U^{+}_{\nu}$, then $u=0$.

For each $\mu \in Q^{+}$, choose a basis $u^{\mu}_{1},u^{\mu}_{2},\ldots,u^{\mu}_{d_{\mu}}$, where $d_{\mu}= \dim(U^{+}_{\mu})$. Then we can take a dual basis $v^{\mu}_{1},v^{\mu}_{2},\ldots,v^{\mu}_{d_{\mu}}$ of $U^{-}_{-\mu}$, i.e. $(v^{\mu}_{i},u^{\mu}_{j})=\delta_{i,j}$. Then we get a basis of $U_{-\nu}^{-}U^{0}U_{\mu}^{+}$ as
$$\{(v^{\nu}_{i}K'^{-1}_{\nu})K'_{\eta}K_{\phi}u^{\mu}_{j}|1\leq i\leq d_{\nu},\ 1\leq j \leq d_{\mu}, \ \text{and}\  \eta,\phi \in Q\}.$$

From the definition of the bilinear form $\langle\cdot|\cdot\rangle$, we obtain
\begin{equation*}
\begin{aligned}
&\ \ \ \ \ \langle(v^{\nu}_{i}K'^{-1}_{\nu})K'_{\eta}K_{\phi}u^{\mu}_{j}\ |\ (v^{\mu}_{k}K'^{-1}_{\mu})K'_{\eta_{1}}K_{\phi_{1}}u^{\nu}_{l
}\rangle\\
&=(v^{\nu}_{i},u^{\nu}_{l})(v^{\mu}_{k},u^{\mu}_{j})(K'_{\eta},K_{\phi_{1}})(K'_{\eta_{1}},K_{\phi})v^{2\rho\cdot\nu}\\
&=\delta_{i,l}\delta_{j,k}(K'_{\eta},K_{\phi_{1}})(K'_{\eta_{1}},K_{\phi})v^{2\rho\cdot\nu}.\\
\end{aligned}
\end{equation*}

Now write $u=\sum_{i,j,\eta,\phi}\theta_{i,j,\eta,\phi}(v^{\nu}_{i}K'^{-1}_{\nu})K'_{\eta}K_{\phi}u^{\mu}_{j}$, and $v=(v^{\mu}_{k}K'^{-1}_{\mu})K'_{\eta_{1}}K_{\phi_{1}}u^{\nu}_{l
}$ with $1\leq k\leq d_{\mu}$ and $1\leq l\leq d_{\nu}$. From the assumption $\langle u|v\rangle=0$, we have
                         $$\sum_{\eta,\phi}\theta_{i,j,\eta,\phi}(K'_{\eta},K_{\phi_{1}})(K'_{\eta_{1}},K_{\phi})v^{2\rho\cdot\nu}=0,$$
for all $1\leq k\leq d_{\mu}$ and $1\leq l\leq d_{\nu}$. That is,
                      $$\sum_{\eta,\phi}\theta_{i,j,\eta,\phi}\chi_{\eta,\phi}v^{2\rho\cdot\nu}=0.$$

It follows from \textrm{lemma 3.6} and linear independence of different characters  that $\theta_{i,j,\eta,\phi}=0$. Hence we get $\langle\cdot|\cdot\rangle$ is non-degenerate.
\end{proof}

\section{The centre of $U_{v,t}$}

We denote the centre of $U_{v,t}$ by $Z(U_{v,t})$.  And since the set $Z(U_{v,t})$  commutes with $K_{i}$, $K_{i}'$ for all $i \in I$, we have$$Z(U_{v,t}) \subseteq U_{0} =U^{0} \bigoplus_{\upsilon >0} U_{- \upsilon}^{-}U^{0}U_{ \upsilon}^{+}.$$

Besides, we define an algebra automorphism $\gamma^{-\rho}: U^{0} \rightarrow U^{0}  $ as
        $$  \gamma^{-\rho}(K_{\eta}K'_{\phi}) =v^{-\rho\cdot(\eta-\phi)}t^{\langle\eta+\phi,\rho\rangle-\langle \rho,\eta+\phi\rangle}K_{\eta}K'_{\phi},$$
        specially, we have $\gamma^{-\rho}(K_{\eta}K'_{-\eta})=v^{-2\rho\cdot \eta}K_{\eta}K'_{-\eta}.$

        The Harish-Chandra map $\xi : Z \rightarrow U^{0}$ is the composition of the two following maps:
$$\gamma^{-\rho} \circ \pi:U_{0}\rightarrow U^{0}\rightarrow U^{0}$$
where   $\pi$ is the canonical projection.

\begin{thm}
 The Harish-Chandra map $\xi$ : Z $\rightarrow$ $U^{0}$ is an injective algebra homomorphism.
\end{thm}
\begin{proof}
   Let $U_{0}=U^{0} \oplus K$,  where $K =\bigoplus_{\nu>0}U_{-\nu}^{-}U^{0}U_{\nu}^{+}$. Since $K$ is the two-sided ideal in $U_{0}$ which is the kernel of $\pi$, and hence of $\xi$, we get $\xi$ is an algebra homomorphism. Assume $z\in Z(U_{v,t})$ and $\xi(z)=0$, to prove $\pi$ is injective, it is enough to show $z$ must be zero. Since $z\in Z(U_{v,t})$ and $\xi(z)=0$,  $z$ must have the form $z=\Sigma_{\nu>0}z_{\nu}$   with    $z_{\nu} \in U_{-\nu}^{-}U^{0}U_{\nu}^{+}$. If $z$ is not zero , we can choose one $\nu \in Q^{+}$/${0}$ minimal with the property that $z_{\nu}\neq 0$. Fix  $\{x_{l}\}$ as bases of  $U_{\nu}^{+}$. We can write $z_{\nu}=\bigoplus_{\{x_{l}\}}\bigoplus_{\eta,\phi \in Q}y_{\eta,\phi,l}K_{\eta}K_{\phi}'x_{l}$,  where $y_{\mu,\phi,l}\in U_{-\nu}^{-}$.
From the fact $\{K_{\eta}K_{\phi}'\}_{\eta,\phi \in Q}$ can form a basis of $U^{0}$ and $\{x_{l}\}$ is a basis of $U_{\nu}^{+}$, there must exist some $x_{l}$ and $\eta'$, $\phi'$ such that
$$y_{\eta',\phi',l'}K_{\eta'}K_{\phi'}'x_{l'}\neq0$$
and $\phi'$ is the maximal one among the nonzero terms.
Since $y_{\eta',\phi',l'}\neq0$, referring to \cite{FL}, lemma 9, there exists some $i \in I$, such that $a_{i}(y_{\eta',\phi',l'})\neq0$.
Write $$0=E_{i}z-zE_{i}=\sum_{\gamma \neq \nu}(E_{i}z_{\gamma}-z_{\gamma}E_{i})+E_{i}z_{\nu}-z_{\nu}E_{i}.$$

Consider all the terms belonging to $\ U^{-}_{-(\nu-i)}U^{0}U^{+}_{\nu}$  and note that $E_{i}z_{\nu}-z_{\nu}E_{i}$ is the only term in $E_{i}z-zE_{i}$ involving $\ U^{-}_{-(\nu-i)}U^{0}U^{+}_{\nu}$. We write

$$E_{i}z_{\nu}-z_{\nu}E_{i}$$
$$=E_{i}(\bigoplus_{\{x_{l}\}}\bigoplus_{\eta,\phi \in Q}y_{\eta,\phi,l}K_{\eta}K_{\phi}'x_{l})-(\bigoplus_{\{x_{l}\}}\bigoplus_{\eta,\phi \in Q}y_{\eta,\phi,l}K_{\eta}K_{\phi}'x_{l})E_{i}.$$

Substitute $x^{-}E_{i}-E_{i}x^{-}=\frac{a_{i}(x)K'_{i}-K_{i}a'_{i}(x)}{v_{i}-v^{-1}_{i}}$ into the above formula, then we must have the conclusion
$$ (v_{i}-v^{-1}_{i})^{-1}\bigoplus_{\{x_{l}\}}\bigoplus_{\eta,\phi \in Q}(a_{i}(y_{\eta,\phi,l})K'_{i}-K_{i}a'_{i}(y_{\eta,\phi,l}))K_{\eta}K_{\phi}'x_{l}=0.$$

Since $\{K_{\eta}K_{\phi}'\}_{\eta,\phi \in Q}$ can form a basis of $U^{0}$ and $\{x_{l}\}$ is a basis of $U_{\nu}^{+}$, we must have $a_{i}(y_{\eta',\phi',l})K_{\eta'}K_{\phi'+i}'x_{l'}$=0.  So $a_{i}(y_{\eta',\phi',l'})=0$, which is a contradiction.

\end{proof}

  Define an algebra homomorphism $\varrho^{\lambda}$ :  $U^{0}$  $\rightarrow$   $\mathbb{Q}(v,t)$ by $$\varrho^{\lambda}(K_{\eta}^{'}K_{\phi})=v^{\lambda\cdot\phi-\eta}t^{\langle\phi+\eta,\lambda\rangle-\langle\lambda,\phi+\eta\rangle},$$
for all $\lambda \in \Lambda$ and $\eta \times \phi \in Q\times Q$.

   Specially, when $\lambda \in Q$, we have
  $$\varrho^{\lambda}(K_{i})=( K_{i}, K_{\lambda}^{'})=v^{\lambda\cdot i}t^{\langle i,\lambda\rangle-\langle\lambda,i\rangle},$$
  $$\varrho^{\lambda}(K_{i}^{'})=( K_{\lambda},  K_{i}^{'})^{-1}=v^{-i\cdot\lambda}t^{\langle i,\lambda\rangle-\langle\lambda,i\rangle}.$$

  Define an algebra homomorphism $\varrho^{0,\lambda}$ :  $U^{0}$  $\rightarrow$   $\mathbb{Q}(v,t)$ by

$$\varrho^{0,\lambda}(K_{i}^{'}K_{j})=v^{2(i+j)\cdot\lambda},$$
for all $\lambda \in \Lambda$ and $i,j \in I$.

    We also define an algebra homomorphism $\varrho^{\lambda,\mu}$ :  $U^{0}$  $\rightarrow$   $\mathbb{Q}(v,t)$ by
    $$\varrho^{\lambda,\mu}(K^{'}_{\eta}K_{\phi})=\varrho^{\lambda}(K^{'}_{\eta}K_{\phi}) \varrho^{0,\mu}(K^{'}_{\eta}K_{\phi}),$$
for all $\lambda \times \mu \in \Lambda \times \Lambda$ and $\eta \times \phi \in Q\times Q$.

\begin{lem}
 Let $u=K_{\eta}^{'}K_{\phi}$, $\eta,\phi \in Q$. If $\varrho^{\lambda,\mu}(u)=1$ for  all $ \lambda,\mu \in \Lambda$, then $u=1$.
\end{lem}

\begin{proof}
  If $\varrho^{0,\lambda}(K_{\eta}^{'}K_{\phi})=0 $ for all $\lambda \in \Lambda$, then $\eta+\phi=0$ from the definition of $\varrho^{0,\lambda}$.  If $\varrho^{\lambda,0}(K_{\eta}^{'}K_{\phi})=0 $ for all $\lambda \in \Lambda$, consider
 $$\varrho^{\lambda}(K_{\eta}^{'}K_{\phi})=v^{\lambda\cdot\phi-\eta}t^{\langle\phi+\eta,\lambda\rangle-\langle\lambda,\phi+\eta\rangle},$$
 since $\lambda \cdot(\eta-\phi)=0$ for all $\lambda \in Q$ and the Cartan matrix  $(a_{ij})$, where $a_{ij}=i\cdot j\ \text{for} \ i,j\in I$, is non-degenerate,  we can conclude $\eta-\phi=0$. Remembering $\eta+\phi=0$, $\eta=\phi=0$ is concluded.
\end{proof}

\begin{lem}
 If $u \in U^{0}$, $\varrho^{\lambda,\mu}(u)=0$, for all $(\lambda,\mu) \in \Lambda \times \Lambda$, then $u=0$.
\end{lem}

\begin{proof}
Corresponding to each $(\eta,\phi) \in Q\times Q$ is the character on the group $\Lambda \times \Lambda$ defined by
                                                           $$(\lambda,\mu) \mapsto \varrho^{\lambda,\mu}(K'_{\eta}K_{\phi}).$$
 It follows from the \textrm{lemma 4.2} that different $(\eta,\phi)$ gives to different characters.

  Suppose now that $u=\sum\theta_{\eta,\phi}K'_{\eta}K_{\phi}$, where $ \theta_{\eta,\phi} \in \mathbb{Q}(v,t)$. By assumption,
                                                           $$\sum\theta_{\eta,\phi} \varrho^{\lambda,\mu}(K'_{\eta}K_{\phi})=0,$$
  for all $(\lambda,\mu) \in \Lambda \times \Lambda$. By the linear independence of different characters, $\theta_{\eta,\phi}=0$ for all $\eta \times \phi \in Q \times Q$ , so $u=0$.
\end{proof}

Define the Weyl group  action on $U^{0}_{\flat}=\bigoplus_{\eta\in Q} \mathbb{Q}(v,t)K'_{\eta}K_{-\eta}$ as follows: $$\sigma (K'_{\eta}K_{-\eta})=K'_{\sigma(\eta)}K_{\sigma(-\eta)}.$$

\begin{thm}
For $U_{v,t}$, $\varrho^{\sigma(\lambda),\mu}(u)= \varrho^{\lambda,\mu}(\sigma^{-1}(u))$, where $u \in U^{0}_{\flat}$, $\sigma \in W$ and $\lambda,\mu \in \Lambda$.
\end{thm}
\begin{proof}
Assume that $u=K'_{\eta}K_{-\eta}$. We only need to check the theorem for each $\sigma_{i}$. Note that $\varrho^{0,\mu}(K'_{\eta}K_{-\eta})=\varrho^{0,\mu}(\sigma^{-1}(K'_{\eta}K_{-\eta}))=1$ from the definition of $\varrho^{0,\mu}$.
 So, we get $\varrho^{\sigma(\lambda),\mu}(u)=\varrho^{\sigma(\lambda),0}(K'_{\eta}K_{-\eta})\varrho^{0,\mu}(K'_{\eta}K_{-\eta})=\varrho^{\lambda,0}(\sigma^{-1}(K'_{\eta}K_{-\eta}))\varrho^{0,\mu}(\sigma^{-1}(K'_{\eta}K_{-\eta}))$
=$\varrho^{\lambda,\mu}(\sigma^{-1}(u))$.
\end{proof}

\begin{lem}
 We set $\iota$  as the unit map of $U_{v,t}$ and $\epsilon$ the co-unit map of $U_{v,t}$. Then   $z \in Z(U_{v,t})$ if and only if $ ad(x)z=(\iota \circ \epsilon)(x)z$ for all $u \in U_{v,t}$.
\end{lem}
\begin{proof}
Let $z \in Z(U_{v,t})$. Then for all $x \in U_{v,t}$, we have
$$ad(x)z=\sum_{(x)}x_{(1)}zS(x_{(2)})=z\sum_{(x)}x_{(1)}S(x_{(2)})=(\iota\circ\varepsilon)(x)z.$$

Conversely, assume that $ad(x)z=(\iota\circ\varepsilon)(x)z$ for all $x \in U_{v,t}$. We can consider it on the generators of $U_{v,t}$. Then
$$K_{i}zK^{-1}_{i}=ad(K_{i})z=(\iota\circ\varepsilon)(K_{i})z=z.$$

Similarly, we have $K'_{i}zK'^{-1}_{i}=z$. Furthermore, we have
$$0=(\iota\circ\varepsilon)(E_{i})z=ad(E_{i})z=E_{i}z-K_{i}zK^{-1}_{i}E_{i}=E_{i}z-zE_{i}$$
and
$$0=(\iota\circ\varepsilon)(F_{i})z=ad(F_{i})z=-zF_{i}(K'_{i})^{-1}+F_{i}z(K'_{i})^{-1}=(-zF_{i}+F_{i}z)(K'_{i})^{-1}.$$
Hence, $z \in Z(U_{v,t}).$
\end{proof}
\begin{lem}
Assume that $\Psi : U^{-}_{-\mu}\times U^{+}_{\nu}\rightarrow \mathbb{Q}(v,t)$ is  a linear map, and let $(\eta,\phi)\in Q\times Q $. Then there exists $u \in U^{-}_{-\upsilon}U^{0}U^{+}_{\mu}$ such that
$$\langle u|(yK'^{-1}_{\mu})K'_{\eta_{1}}K_{\mu_{1}}x\rangle=(K'_{\eta_{1}},K_{\phi})(K_{\eta}',K_{\phi_{1}})\Psi(y,x)$$
for all $x \in U^{+}_{\nu}\ ,\ y \in U^{-}_{-\mu}$ and $(\eta_{1},\phi_{1}) \in Q\times Q$.
\end{lem}
\begin{proof}
For each $\mu \in Q^{+}$, we choose a basis $u^{\mu}_{1},u^{\mu}_{2},\ldots,u^{\mu}_{d_{\mu}}$ ($d_{\mu}=\dim(U^{+}_{\mu})$) of $U^{+}_{\mu}$ and a dual basis  $v^{\mu}_{1},v^{\mu}_{2},\ldots,v^{\mu}_{d_{\mu}}$ of $U^{-}_{-\mu}$ such that $(v^{\mu}_{i},u^{\mu}_{j})=\delta_{i,j}$. If we set
$$u=\sum_{(i,j)}\Psi(v^{\mu}_{j},u^{\nu}_{i})v^{\nu}_{i}(K'_{\nu})^{-1}K'_{\eta}K_{\phi}u^{\mu}_{j}v^{-2\rho\cdot \nu},$$
then it is straightforward to verify that $u$ satisfies the requirement.
\end{proof}
We define a $U_{v,t}$-module structure on the dual space $U^{*}_{v,t}$ by setting $(x.f)(v)=f(S(x)v)$, for any $u,v \in U_{v,t}, f \in U_{v,t}^{\ast}$. We also define a map $\beta: U_{v,t} \rightarrow U_{v,t}^{\ast}$ by setting $$\beta(u)(v)=\langle u | v\rangle \ \ \forall u,v \in U_{v,t}.$$

Considering that $U_{v,t}$-module structure on $U_{v,t}$ is given by the adjoint representation. We get $\beta$ is an injective $U_{v,t}$-module homomorphism because of the ad-invariance and the non-degeneracy of the bilinear form $\langle ,\rangle$.

\begin{dfn} Let $M$ be a finite dimensional module of $U_{v,t}$. For each $m \in M$ and $f \in M^{*}$, we define an element $C_{f,m} \in U_{v,t}^{*}$ by $C_{f,m}(v)$ =$f(v.m)$, for any $v \in U_{v,t}$.
\end{dfn}

\begin{prop}
Assume that $M$ is a finite $U_{v,t}$-module, $M=\bigoplus_{\lambda \in wt(M)}M_{\lambda}$. Suppose that the weight set $wt(M)$ of $M$ satisfies:\ $wt(M) \in Q$. For each $f \in M^{*}$,\ $m \in M$, there exist a unique $u \in  U_{v,t}$ such that

\ \ \ \ \ \ \ \ \ \ \ \ \ \ \ \ \ \ \ \ \ \ \ \ \ \ \ \ \ \ \ \ \ \ $C_{f,m}(v)=\langle u|v\rangle$ \ \ \ \ \ for all $v \in U_{v,t}$.
\end{prop}
\begin{proof}
Let $m \in M_{\lambda}$, $v=(yK'^{-1}_{\mu})K'_{\eta_{1}}K_{\phi_{1}}x$, where $x \in U^{+}_{\nu}$,\ $y \in U^{-}_{-\mu}$,\ $(\eta_{1},\phi_{1}) \in Q\times Q$. Then for each $f \in M^{*}$, we have

\begin{equation*}
\begin{aligned}
C_{f,m}(v)&=C_{f,m}((yK'^{-1}_{\mu})K'_{\eta_{1}}K_{\phi_{1}}x)=f((yK'^{-1}_{\mu})K'_{\eta_{1}}K_{\phi_{1}}x.m)\\
&=\varrho^{\nu+\lambda}(K'_{\eta_{1}}K_{\phi_{1}})f((yK'^{-1}_{\mu})x.m).\\
\end{aligned}
\end{equation*}

Note that $(y,x)\mapsto f(yx.m)$ is bilinear, and $\varrho^{\nu+\lambda}(K'_{\eta_{1}})=(K'_{\eta_{1}},K_{-\nu-\lambda})$, $\varrho^{\nu+\lambda}(K_{\phi_{1}})=(K'_{\nu+\lambda},K_{\phi_{1}})$. By \textrm{lemma 4.6}, there exist a unique $u_{\nu\mu}$ such that $C_{f,m}(v)=\langle u_{\nu\mu},v\rangle$ for all $v \in U^{-}_{-\mu}U^{0}U^{+}_{\nu}$.
Now for arbitrary $v \in U_{v,t}$, we write $v=\sum_{\mu,\nu}v_{\mu\nu}$ where $v_{\mu\nu} \in U^{-}_{-\mu}U^{0}U^{+}_{\nu}$. Since $M$ is finite-dimensional, there exists a finite set $\Upsilon$ of pair $Q \times Q$, such that
$$C_{f,m}(v)=C_{f,m}(\sum_{(\mu,\nu)\in \Upsilon}v_{\mu\nu}).$$

Setting $u=\Sigma_{\mu,\nu}u_{\mu\nu}$, we have

$$C_{f,m}(v)=C_{f,m}(\sum_{(\mu,\nu)\in \Upsilon}v_{\mu\nu})=\sum_{(\mu,\nu)\in \Upsilon}C_{f,m}(v_{\mu\nu}) $$
$$=\sum_{(\mu,\nu)\in \Upsilon}\langle u_{\mu\nu}|v_{\mu\nu}\rangle=\langle u|v\rangle.$$

The uniqueness follows from \textrm{proposition 4.7}.
\end{proof}

Assume $M$ is a finite dimentional $U_{v,t}$-module, and we define a linear map $\Theta$ on $M$ by $$\Theta(m_{\lambda})=v^{-2\rho \cdot \lambda}m_{\lambda},$$ for all $m \in M_{\lambda}, \lambda \in \Lambda$.

\begin{lem}
 $\Theta u = S^{2}(u) \Theta$, for all $u \in U_{v,t}.$
\end{lem}
\begin{proof}
 We have only to check it holds for generators $E_{i}, F_{i}, K_{i}$ and $K_{i}^{'}$. Also we fix one $m_{\lambda} \in M_{\lambda}:$

 (1)When $u=K_{i}$. $S^{2}(K_{i})\Theta(m_{\lambda})=K_{i}\Theta(m_{\lambda})=\Theta K_{i}(m_{\lambda}).$

 (2)When $u=K'_{i}$. $S^{2}(K'_{i})\Theta(m_{\lambda})=K'_{i}\Theta(m_{\lambda})=\Theta K'_{i}(m_{\lambda}).$

 (3)When $u=E_{i}$.
 $S^{2}(E_{i})\Theta(m_{\lambda})=K^{-1}_{i}E_{i}K_{i}\Theta(m_{\lambda})$
 =$v^{i\cdot i}E_{i}\Theta(m_{\lambda})$
 =$v^{-2\rho \cdot (\lambda+i)}E_{i}m_{\lambda}$
 =$\Theta(E_{i}.m_{\lambda}).$

 (4)When $u=F_{i}$.
 $S^{2}(F_{i})\Theta(m_{\lambda})=K'_{i}F_{i}K'^{-1}_{i}\Theta(m_{\lambda})$
 =$v^{i \cdot i}F_{i}\Theta(m_{\lambda})$
 =$v^{-2\rho \cdot (\lambda-i)}F_{i}m_{\lambda}$
 =$\Theta(F_{i}.m_{\lambda}).$

\end{proof}

For $\lambda\in \Lambda$, define $f_{\lambda} \in U_{v,t}^{*}$  by $f_{\lambda}(u)=tr_{L( \lambda)}(u\Theta)$ for each $u \in U_{v,t}$.

\begin{lem} If $\lambda \in \Lambda^{+}\cap Q$, then $f_{\lambda} \in \text{Im}(\beta)$.
\end{lem}

\begin{proof} We have only to find an element $u \in U_{v,t}$ such that $\langle u |v \rangle$=$f_{\lambda}(v)$, for any $v \in U_{v,t}$. Let $\lambda \in \Lambda^{+} \cap Q$ and $k$=dim$(L_{\lambda})$. Denote by $\{m_{i}\}_{i=1,2,...,k}$ a basis of $L_{\lambda}$ and $\{f_{i}\}_{i=1,2,...,k}$  the dual basis of $L_{\lambda}^{*}$. Then by definition we have $$f_{\lambda}(v)=tr_{L(\lambda)}(v \Theta).$$

Since $\Theta$ is invertible, $\{\Theta m_{i}\}_{i=1,2,...,k}$ form a basis of $L_{\lambda}$
$$tr_{L(\lambda)}(v \Theta)=\sum_{i=1}^{k}f_{i}(v \Theta m_{i})=\sum_{i=1}^{k}C_{f_{i},\Theta m_{i}}(v)= \sum_{i=1}^{k}\langle u_{i}| v \rangle = \langle \sum_{i=1}^{k}u_{i} |v \rangle.$$
Set $u=\sum_{i=1}^{k}u_{i}$, $u$ is the element we want.
\end{proof}
\begin{thm} For $\lambda \in \Lambda^{+} \cap Q$, $z_{\lambda} = \beta^{-1}( f_{\lambda}) \in Z(U_{v,t})$.
\end{thm}
\begin{proof} For all $x \in U_{v,t}$, we have

     \begin{equation*}
\begin{aligned}
(S^{-1}(x)f_{\lambda})(u)&=f_{\lambda}(ad(x)u)  \\
     &=tr_{L(\lambda)}(\sum_{(x)}x_{(1)}uS(x_{(2)})\Theta)\\
     &=tr_{L(\lambda)}(u\sum_{(x)}S(x_{2})\Theta x_{(1)})\\
     &=tr_{L(\lambda)}(u\sum_{(x)}S(x_{2})S^{2}(x_{(1)})\Theta)\\
     &=tr_{L(\lambda)}(uS(\sum_{(x)}S(x_{(1)})x_{(2)})\Theta)\\
     &=(\iota \circ \epsilon)(x)tr_{L(\lambda)}(u \Theta)\\
     &=(\iota \circ \epsilon)(x)f_{\lambda}(u).
\end{aligned}
\end{equation*}

     Substituting $x$ for $S^{-1}(x)$ in the above and noticing that $\varepsilon \circ S = \varepsilon $, we can conclude $$x.f_{\lambda}=(\iota \circ \epsilon)(x)f_{\lambda}.$$

     Now let
     $$x f_{\lambda}=x\beta(\beta^{-1}(f_{\lambda}))=\beta(ad(S(x))\beta^{-1}(f_{\lambda})),$$
     $$(\iota \circ \epsilon )(x)f_{\lambda}=(\iota \circ \epsilon )(x)\beta (\beta^{-1}(f_{\lambda}))=\beta ((\iota \circ \epsilon)(x) \beta^{-1}f_{\lambda}).$$

Since $\beta$ is an injective map, for all $x \in U_{v,t}$
      $$ad(x)(\beta^{-1}(f_{\lambda}))= (\iota \circ \epsilon )(x)\beta^{-1}(f_{\lambda}), $$
      then we have $\beta^{-1}(f_{\lambda}) \in Z(U_{v,t})$.
\end{proof}

      Recall that $U_{\flat}^{0}=\bigoplus_{\eta \in Q}\mathbb{Q}(v,t)K'_{\eta}K_{-\eta}$ and define $(U_{\flat}^{0})^{W}=\{u \in U_{\flat}^{0}|\sigma(u)=u, \forall \sigma \in W\}$

\begin{thm} For $U_{v,t}$, $(U_{\flat}^{0})^{W} \subseteq \xi(Z(U_{v,t}))$.
\end{thm}

\begin{proof}  Set $z_{\lambda}=\beta^{-1}(f_{\lambda})$ for $\lambda \in \Lambda^{+}\cap Q$ and write$$z_{\lambda}=\sum_{\upsilon \geq 0}z_{\lambda,\upsilon} ,$$  and specially, $$z_{\lambda,0}=\sum_{\eta,\phi \in Q\times Q }\theta_{\eta,\phi}K'_{\eta}K_{\phi},$$where $z_{\lambda,\upsilon} \in U^{-}_{-\upsilon}U^{0}U^{+}_{\upsilon}$ and $\theta_{\eta,\phi} \in \mathbb{Q}(v,t)$. Now for $(\eta_{1},\phi_{1}) \in Q \times Q$, we have
$$\langle z_{\lambda} | K'_{\eta_{1}}K_{\phi_{1}}\rangle=\langle z_{\lambda,0} | K'_{\eta_{1}}K_{\phi_{1}}\rangle=\sum_{(\eta,\phi)}\theta_{\eta,\phi}(K'_{\eta_{1}},K_{\phi})(K'_{\eta},K_{\phi_{1}}).$$

On the other hand, from the definition of $z_{\lambda}$, we also have
\begin{equation*}
\begin{aligned}
\langle z_{\lambda}| K'_{\eta_{1}}K_{\phi_{1}}\rangle&=\beta(z_{\lambda})(K'_{\eta_{1}}K_{\phi_{1}})=f_{\lambda}(K'_{\eta_{1}}K_{\phi_{1}})=tr_{L(\lambda)}(K'_{\eta_{1}}K_{\phi_{1}}\Theta)\\
&=\sum_{\mu \leq \lambda}v^{-2\rho \cdot \mu}\dim(L(\lambda)_{\mu})\varrho^{\mu}(K'_{\eta_{1}}K_{\phi_{1}})\\
&=\sum_{\mu \leq \lambda}v^{-2\rho \cdot \mu}\dim(L(\lambda)_{\mu}) (K'_{\eta_{1}},K_{-\mu}) (K'_{\mu},K_{\phi_{1}}).\\
\end{aligned}
\end{equation*}

Now we get the following equation,

$$\sum_{(\eta,\phi)}\theta_{\eta,\phi}\chi_{\eta,\phi}=\sum_{\mu \leq \upsilon}v^{-2\rho \cdot \lambda}\dim(L(\lambda)_{\mu})\chi_{\mu,-\mu}.$$

Since these different characters are linearly independent, we have
(1)\ $\theta_{\eta,\phi}=\dim(L(\lambda))v^{-2\rho\cdot\lambda}m_{\lambda}$,  when $(\eta,\phi)=(\mu,-\mu).$
(2)\  $\theta_{\eta,\phi}=0$, otherwise.

So, we have

$$z_{\lambda,0}=\sum_{\mu \leq \lambda}v^{-2\rho \cdot \mu}m_{\lambda}\dim(L(\lambda)_{\mu})K'_{\mu}K_{-\mu}.$$

Then, we obtain
$$\xi(z_{\lambda})=\varrho^{-\rho}(z_{\lambda,0})=\sum_{\mu \leq \lambda}\dim(L(\lambda)_{\mu})K'_{\mu}K_{-\mu}.$$

For $\lambda \in \Lambda^{+}\cap Q$, we define

         $$av(\lambda)=\frac{1}{|W|}\sum_{\sigma \in W}\sigma(K'_{\lambda}K_{-\lambda})=\frac{1}{|W|}\sum_{\sigma \in W}K'_{\sigma(\lambda)}K_{-\sigma(\lambda)}.$$

For each $\eta\in Q$, there exists a unique $\sigma \in W$ such that $\sigma(\eta)\in \Lambda^{+}\cap Q$, so $\{av(\lambda)|\lambda \in \Lambda^{+}\cap Q\}$  forms a basis of $(U_{\flat}^{0})^{W}$. So we only have to show $av(\lambda) \in \text{Im}(\xi)$. By induction on the height of $\lambda$,  if $\lambda=0$, $av(\lambda)=1=\xi(1)$. Assume that $\lambda > 0$. Since $\dim L(\lambda)_{\mu}=\dim L(\lambda)_{\sigma(\mu)}$, for all $\sigma \in W$, and $\dim(L(\lambda)_{\lambda})=1$, we can obtain

  $$\xi(z_{\lambda})=|W|av(\lambda)+|W|\sum \dim(L(\lambda)_{\mu})av(\mu),$$
where the sum is over $\mu$  such that $\mu < \lambda$ and $\mu \in \Lambda^{+}\cap Q$. By the induction hypothesis, we get $av(\lambda) \in \text{Im}(\xi)$. So $(U^{0}_{\flat})^{W}\subseteq \xi(Z(U_{v,t}))$.

\end{proof}

\section{The image of the Harish-Chandra homomorphism of $U_{v,t}(g)$ }

 Let us recall the connection between the one-parameter quantum group and the two-parameter quantum group in \cite{FL}.

For any $\gamma=(\gamma_{1},\gamma_{2})$, $\eta=(\eta_{1},\eta_{2})$ $\in$ $\mathbb{Z}^{I} \times \mathbb{Z}^{I}$, we define a bilinear form on $\mathbb{Z}^{I} \times \mathbb{Z}^{I}$ by
$$[\gamma,\eta]'=[\gamma_{2},\eta_{2}]-[\gamma_{1},\eta_{1}].$$
The algebra $U_{v,t}$ admits a $\mathbb{Z}^{I} \times \mathbb{Z}^{I}$-grading  by defining the degrees of generators as follows,
 $$\text{deg}(E_{i})=(i,0), \text{deg}(K_{i})=(i,i)=\text{deg}(K'_{i})$$
 $$\text{deg}(F_{i})=(0,i), \text{deg}(K^{-1}_{i})=(-i,-i)=\text{deg}(K'^{-1}_{i}) .$$

 And in \cite{FL}, the authors define a new multiplication $``\ast"$ on $U_{v,t}$ by  
 $$x \ast y=t^{-[|x|,|y|]'}xy,$$
 for any homogeneous elements $x,y \in U_{v,t}$. Since $[,]'$ is a bilinear form, $(U_{v,t},\ast)$ is an associative algebra over $\mathbb{Q}(v,t)$. Define the multiplication, denoted by $\ast$, on $U_{v,t}\otimes U_{v,t}$ by
 $$(x\otimes y)\ast(x'\otimes y')=x\ast x'\otimes y\ast y'.$$
 This gives a algebra structure on $U_{v,t}\otimes U_{v,t}$. $(U_{v,t},\ast)$ has a Hopf algebra structure with the comultiplication $\bigtriangleup^{\ast}$, the counit $\varepsilon^{\ast}$ and the  antipode $S^{\ast}$. The image of generators $E_{i}$,$F_{i}$,$K_{i}$ and $K'_{i}$ under the map $\bigtriangleup^{\ast}$  (resp.  $\varepsilon^{\ast}$  and $S^{\ast}$) are the same as
the ones under the map $\bigtriangleup$  (resp.  $\varepsilon$  and $S$).

\begin{thm}   \cite[Theorem 4]{FL} If $(I ,\cdot )$ is the Cartan datum associated to $\Omega$, then there is a Hopf-algebra isomorphism
                                 $$(U_{v,t},\ast,\bigtriangleup^{\ast},\varepsilon^{\ast},S^{\ast})\simeq(U_{v,t}(I, \cdot),\cdot,\triangle_{1},\varepsilon_{1},S_{1})$$
sending the generators in $U_{v,t}$ to the respective generators in $U_{v,t}(I, \cdot)$, where $\mathbb{Q}(v,t)\otimes U_{v}(I,\cdot)\cong(U_{v,t}(I,\cdot), \ast)$, and $U_{v}(I,\cdot)$ is the one-parameter case with Cartan datum $(I ,\cdot )$.

\end{thm}

 In the rest of this section, grading on $U_{v,t}$ refers to the $\mathbb{Z}^{I} \times \mathbb{Z}^{I}$-grading $(\cdot ,\cdot )$. In this case, homogeneous central elements can be defined with degree $(\eta,\eta)$ for some $\eta \in \mathbb{Z}^{I}$  from the fact that central elements commute with $K_{i}$ and $K'_{i}$ for all $i \in I$. This grading allows us to consider homogeneous central elements  for convenience.

\begin{thm}
The two following propositions are equivalent :

\ $(a)$ There doesn't exist any nonzero $\eta \in Q$, such that $[\eta,i]-[i,\eta]=0$ for all $i \in I$.

\ $(b)$ $\xi(Z(U_{v,t}))=(U^{0}_{\flat})^{W}$.
\end{thm}
\begin{proof}

 Assume $(b)$ holds. If there is some nonzero $\eta' \in \mathbb{Z}^{I}$, such that $[\eta',i]-[i,\eta']=0$ for all $i \in I$, then  consider the commutation relation of the element $K'_{\eta'}K_{\eta'}$. For all $i \in I$,
$$K'_{\eta'}K_{\eta'}E_{i}=t^{\langle 2\eta',i\rangle-\langle i,2\eta'\rangle}E_{i}K'_{\eta'}K_{\eta'}=t^{[i,2\eta']-[2\eta',i]}E_{i}K'_{\eta'}K_{\eta'}=E_{i}K'_{\eta'}K_{\eta'},$$
$$K'_{\eta'}K_{\eta'}F_{i}=t^{\langle i,2\eta'\rangle-\langle2\eta',i\rangle}F_{i}K'_{\eta'}K_{\eta'}=t^{[2\eta',i]-[i,2\eta']}F_{i}K'_{\eta'}K_{\eta'}=F_{i}K'_{\eta'}K_{\eta'}.$$
So we have $K'_{\eta'}K_{\eta'} \in Z(U_{v,t})$, which is contradictory to $(b)$.

Conversely, assume $(a)$ holds. Since we have proved $(U^{0}_{\flat})^{W}\subseteq\xi(Z(U_{v,t}))$ in theorem 4.12, then $\xi(Z(U_{v,t}))=(U^{0}_{\flat})^{W}$ is equivalent to $(U^{0}_{\flat})^{W}\supseteq\xi(Z(U_{v,t}))$. Parallel to the one parameter case, the images of Harish-Chandra homomorphism of the central elements with degree $(0,0)$ are in  $(U^{0}_{\flat})^{W}$.  So, to show $(U^{0}_{\flat})^{W}\supseteq\xi(Z(U_{v,t}))$, it is enough to show there is no central element $z$ with degree $(\eta,\eta)$, where $\eta\neq 0$ and $\eta$ is under the assumption $(a)$. If there exists some homogeneous central element $u$ with degree $|u|=(\eta,\eta)$, $\eta \in \mathbb{Z}^{I}\setminus\{0\}$, we have
$$u\in \bigoplus_{\alpha \in Q^{+}} U_{(0,\alpha)}^{-}U^{0}_{(\eta-\alpha,\eta-\alpha)}U^{+}_{(\alpha,0)}.$$

Choose  $\gamma \in Q^{+}$ as the maximal one in $u$ such that $ U_{(0,\gamma)}^{-}U^{0}_{(\eta-\gamma,\eta-\gamma)}U^{+}_{(\gamma,0)}$ are nonzero. Also choose a basis $\{x^{\gamma}_{1},x^{\gamma}_{2},\ldots,x^{\gamma}_{d_{\gamma}}\}$ of $ U_{(0,\gamma)}^{-}$ where $d_{\gamma}$ denote the dimension of $U_{(0,\gamma)}^{-}$.
 From the triangle decomposition, we have

$$u=\sum^{d_{\gamma}}_{j=1}\theta_{j}x^{\gamma}_{j}(\sum_{a+b=\eta-\gamma}K_{a}'K_{b}U^{+}_{(\gamma,0)})+\sum_{\alpha\neq\gamma}U_{(0,\alpha)}^{-}U^{0}_{(\eta-\alpha,\eta-\alpha)}U^{+}_{(\alpha,0)}.$$
where $\theta_{j} \in \mathbb{Q}(v,t)$.
Under the assumption $(a)$, there is some $i \in I$ satisfying $[\eta,i]-[i,\eta]\neq0$. Then $E_{i}u=uE_{i}$ can be expressed as
\begin{equation*}
\begin{aligned}
E_{i}(\sum^{d_{\gamma}}_{j=1}\theta_{j}x^{\gamma}_{j}(\sum_{a+b=\eta-\gamma}K_{a}'K_{b}U^{+}_{(\gamma,0)})+\sum_{\alpha\neq\gamma}U_{(0,\alpha)}^{-}U^{0}_{(\eta-\alpha,\eta-\alpha)}U^{+}_{(\alpha,0)})\\
=(\sum^{d_{\gamma}}_{j=1}\theta_{j}x^{\gamma}_{j}(\sum_{a+b=\eta-\gamma}K_{a}'K_{b}U^{+}_{(\gamma,0)})+\sum_{\alpha\neq\gamma}U_{(0,\alpha)}^{-}U^{0}_{(\eta-\alpha,\eta-\alpha)}U^{+}_{(\alpha,0)})E_{i}.
\end{aligned}
\end{equation*}

Note that $E_{i}U_{(0,\gamma)}^{-}U^{0}_{(\eta-\gamma,\eta-\gamma)}U^{+}_{(\gamma,0)}$ is the only term of $E_{i}u$ involving $U_{(0,\gamma)}^{-}U^{0}_{(\eta-\gamma,\eta-\gamma)}U^{+}_{(\gamma+i,0)}$ because of the maximum of $\gamma$. For the same reason, $U_{(0,\gamma)}^{-}U^{0}_{(\eta-\gamma,\eta-\gamma)}U^{+}_{(\gamma,0)}E_{i}$ is the only term of $uE_{i}$ involving $U_{(0,\gamma)}^{-}U^{0}_{(\eta-\gamma,\eta-\gamma)}U^{+}_{(\gamma+i,0)}$.

And
\begin{equation*}
\begin{aligned}
&\ \ \ \ \ \ \ \ \ \ \ \  \ E_{i}(\sum^{d_{\gamma}}_{j=1}\theta_{j}x^{\gamma}_{j}(\sum_{a+b=\eta-\gamma}K_{a}'K_{b}U^{+}_{(\gamma,0)})\\
&=\sum^{d_{\gamma}}_{j=1}\theta_{j}x^{\gamma}_{j}E_{i}(\sum_{a+b=\eta-\gamma}K_{a}'K_{b}U^{+}_{(\gamma,0)})+U_{(0,\gamma-i)}U^{0}U_{(\gamma,0)}\\
&=\sum^{d_{\gamma}}_{j=1}v^{(b-a)\cdot i}t^{\langle i,\eta-\gamma\rangle-\langle\eta-\gamma,i\rangle}\theta_{j}x^{\gamma}_{j}(\sum_{a+b=\eta-\gamma}K_{a}'K_{b}E_{i}U^{+}_{(\gamma,0)})+U_{(0,\gamma-i)}U^{0}U_{(\gamma,0)}.\\
\end{aligned}
\end{equation*}

 Choose some $\theta_{k}\neq0$, $K_{a'}$ and $K'_{b'}$, since $E_{i}u=uE_{i}$, we must have
$$v^{(b'-a')\cdot i}t^{\langle i,\eta-\gamma\rangle-\langle\eta-\gamma,i\rangle}\theta_{k}x^{\gamma}_{k}K_{a'}'K_{b'}E_{i}U^{+}_{(\gamma,0)}=\theta_{k}x^{\gamma}_{k}K_{a'}'K_{b'}U^{+}_{(\gamma,0)}E_{i}.$$

 Since $\{x^{\gamma}_{1},x^{\gamma}_{2},\ldots,x^{\gamma}_{d_{\gamma}}\}$ is a basis of $U_{(0,\gamma)}^{-}$, we obtain

 $$v^{(b'-a')\cdot i}t^{\langle i,\eta-\gamma\rangle-\langle\eta-\gamma,i\rangle}E_{i}U^{+}_{(\gamma,0)}=U^{+}_{(\gamma,0)}E_{i}.$$

 Rewrite the above equation under the multiplication $\ast$, we have
$$v^{(b'-a')\cdot i}t^{\langle i,\eta\rangle-\langle\eta,i\rangle}E_{i}\ast U^{+}_{(\gamma,0)}=U^{+}_{(\gamma,0)}\ast E_{i}.$$

 Here the element $U^{+}_{(\gamma,0)}$ can be expressed as the finite sum $\theta'\sum_{i\in \mathbb{Z}} t^{i}\otimes x_{i}$ where $\theta' \in \mathbb{Q}(v,t)$  and $x_{i}\in U_{v}$, and note that only finite $x_{i}$ are nonzero.
 From the equation
 $$v^{(b'-a')\cdot i}t^{\langle i,\eta\rangle-\langle\eta,i\rangle}E_{i}\ast U^{+}_{(\gamma,0)}=U^{+}_{(\gamma,0)}\ast E_{i},$$
 we know  ${\langle i,\eta\rangle-\langle\eta,i\rangle}$ must be $0$. It is a contradiction.

\end{proof}

Now we would like to consider some subalgebras of $U_{v,t}(g)$  and compare the images of Harish-Chandra homomorphism.
For $J \subseteq I$, let $U_{J}$ be the subalgebra of $U_{J}=\langle U^{0}_{v,t},E_{i},F_{i} | i\in J\rangle$ generated by $U^{0}_{v,t}$ and $E_{i}$, $F_{i}$ with $i \in J$. We denote
 by $Z_{J}$ the centre of $U_{J}$ and  $ \xi_{J} $ the Harish-Chandra homomorphism for $U_{J}$. Then we would like to show $\text{Im}(\xi)\subset \text{Im}(\xi_{J})$.
Denote by $U^{+}_{J}$ the subalgebra of $U_{J}$ generated by  $E_{i}$ with $i \in J$  and   $U^{-}_{J}$ the subalgebra of $U_{J}$ generated by $F_{i}$ with $i \in J$. Recall  the skew-Hopf pair $(\cdot|\cdot)$  on $U_{v,t}$. Set
$$R^{+}_{J}=\{x \in U^{+}_{v,t}|(x|U_{J}^{-})=0\}=\{x \in U^{+}_{v,t}|(x|U_{J}^{-}U_{v,t}^{0})=0\},$$
$$R^{-}_{J}=\{y \in U^{-}_{v,t}|(U^{+}_{j}|y)=0\}=\{y \in U^{-}_{v,t}|(U_{v,t}^{0}U^{+}_{J}|y)=0\},$$
$$R_{J}=R^{-}_{J}U^{0}_{v,t}U^{+}_{v,t}+U^{-}_{v,t}U^{0}_{v,t}R^{+}_{J}.$$

\begin{lem}
We have
\begin{equation*}
\begin{aligned}
&(1)\ \ U_{v,t}=U_{J}\oplus R_{J},\\
&(2)\ \ U_{J}R_{J}U_{J}\subset R_{J},\\
&(3)\ \ (\epsilon \otimes 1 \otimes \epsilon )(R_{J})=0.\\
\end{aligned}
\end{equation*}

\end{lem}

\begin{proof}

(1). \ It suffices to show $U^{+}_{\gamma}=U^{+}_{J,\gamma} \oplus R^{+}_{J, \gamma}$  for any $\gamma \in Q^{+}$. From the definition of $R_{J}^{+}$, we get the following equation,
$$R^{+}_{J,\gamma}=\text{Ker}(U_{\gamma}^{+} \rightarrow (U^{-}_{-\gamma})^{\ast} \rightarrow (U^{-}_{J,-\gamma})^{\ast}), $$
where $U_{\gamma}^{+} \rightarrow (U^{-}_{-\gamma})^{\ast}$ is an algebra isomorphism because of the fact that skew-Hopf pair is nondegenerate on $U^{+}_{\gamma}\times U^{-}_{-\gamma}$ and $(U^{-}_{-\gamma})^{\ast} \rightarrow (U^{-}_{J,-\gamma})^{\ast}$ is the canonical  projection.

From the above equation and the epimorphism of $U_{\gamma}^{+} \rightarrow (U^{-}_{-\gamma})^{\ast}$, we have the following dimensional equation,
$$\dim R^{+}_{J, \gamma}=\dim U^{+}_{\gamma}-\dim U^{-}_{J,\gamma}=\dim U^{+}_{\gamma}-\dim U^{+}_{J,\gamma}. $$
Since $(\cdot |\cdot )$ is nondegenerate on $U_{J,\gamma}^{+} \times U^{-}_{J,-\gamma}  $, we have $R^{+}_{J,\gamma}\cap U^{+}_{J,\gamma}= \{0 \}$.

(2). \ For one $x^{-}_{J} \in U^{-}_{J}$ , from \textrm{proposition 3.1}  and \textrm{lemma 3.2}, we  get

  $$(E_{i}x | x^{-}_{J})=(E_{i}|F_{i}K'_{|x^{-}_{J}|-i})(x|a'_{i}(x_{J}^{-})),\ \text{for all}\  x \in R^{+}_{J}\ \text{and}\ i \in I.$$
Since $a'_{i}(x_{J}^{-}) \in U^{-}_{J}$, we obtain  $(E_{i}x | x^{-}_{J})=0$.

  Similarly,
  $$(x E_{i}| x^{-}_{J})=(E_{i}|F_{i})(x|a_{i}(x_{J}^{-})K'_{i}),\ \text{for all}\  x \in R^{+}_{J}.$$
Since $a_{i}(x_{J}^{-}) \in U^{-}_{J}$, we obtain  $(x E_{i}| x^{-}_{J})=0$.

  So we can conclude that $R^{+}_{J}$ is a two-sided ideal of $U^{+}_{v,t}$. By analogue with the $R^{+}_{J}$ case, $R^{-}_{J}$ is a two-sided ideal of $U^{-}_{v,t}$. Hence it suffices to show
  $$U^{+}_{J}R^{-}_{J} \subset R^{-}_{J}U_{v,t}, \ \ \ \ \ R^{+}_{J}U^{-}_{J} \subset U_{v,t}R^{+}_{J}.$$

  We show $U^{+}_{J}R^{-}_{J} \subset R^{-}_{J}U_{v,t}$ for example, and the other case can be proved in a similar way.

  Set any $y \in R^{-}_{J}$ and  $j \in J$,  from \textrm{lemma 3.2}, we can get

  $$yE_{j}-E_{j}y=\frac{a_{j}(y)K'_{j}-K_{j}a'_{j}(y)}{v_{j}-v^{-1}_{j}}.$$

  Since $(E_{j}K'_{|y|-j}U^{+}_{J}|y)=(E_{j}K_{|y|-j}|F_{j}K'_{|y|-j})(U^{+}_{J}|a'_{j}(y))$ and $E_{j}U^{+}_{J} \in U^{+}_{J}$, we have $(E_{j}U^{+}_{J}|y)=0$. Then $(U^{+}_{J}|a'_{j}(y))=0 $ because $(E_{j}|F_{j})\neq 0$. That is to say, $a'_{j}(y)$ belongs to $R^{-}_{J}$. Similarly, $a_{j}(y)$ belongs to $R^{-}_{J}$. (2) is proved.

  (3). Clear.
\end{proof}

\begin{prop}
$\text{Im}(\xi)\subseteq \text{Im}(\xi_{J})$
\end{prop}
\begin{proof}
Let $z \in Z(U_{v,t})$ and write $z=z_{1}+z_{2}$ with $z_{1} \in U_{J}$, $z_{2} \in R_{J}$. \textrm{ lemma 5.3 (2)}, $z_{1} \in Z(U_{J})$, and hence by \textrm{By lemma 5.3 (3)}, $\xi(z)=\xi_{J}(z_{1}) \in Im(\xi_{J})$.
\end{proof}

With different choices of the set $J$, we shall get different results. Now consider a special case when $J={i}$  for one $i \in I$, then the algebra $U_{J_{i}} $ is generated by $E_{i}$, $F_{i}$ for this fixed $i$ and $K^{\pm 1}_{j}$, $K'^{\pm 1}_{j}$ for all $j \in I$. And let $U_{v,t}(g_{_{i}})$ be the subalgebra generated by $E_{i}$, $F_{i}$, $K^{\pm 1}_{i}$, $K'^{\pm 1}_{i}$ for this fixed $i \in I$.

The centre of $U_{v,t}(g_{_{i}})$ has a basis of monomials $X^{m}Y^{n}$, $m \in \mathbb{Z}$, $n \in \mathbb{Z}_{\geq0}$, where $X=K_{i}K'_{i}$  and $Y$ is the Casimir element,
            $$Y=E_{i}F_{i}+\frac{v^{-1}K_{i}+vK'_{i}}{(v-v^{-1})^{2}}=F_{i}E_{i}+\frac{vK_{i}+v^{-1}K'_{i}}{(v-v^{-1})^{2}}.$$

 The following theorem gives the description of the centre of $U_{J_{i}}$ for any fixed $i \in I$.
\begin{thm}
For any $i \in I$, denote by $Q'$ the $\mathbb{Z}$-lattice spanned by $I/\{i\}$. Then any $u \in  Z(U_{J_{i}})$   can be expressed as the sum
$$u=\bigoplus_{x \in Q', y \in Q'}\mathbb{Q}(v,t) K_{x}K'_{y}L,$$
where $L \in U_{v,t}(g_{_{i}})$ and $K_{x}$, $K'_{y}$ and $L$ satisfy the following two conditions:

$(1)$ we require $K_{x}$ and $K'_{y}$ to meet the condition: $E_{i}K_{x}K'_{y}=v^{-2k}K_{x}K'_{y}E_{i}$ for some $k \in \mathbb{Z}$.

$(2)$ $L=\{K^{k}_{i}(\Sigma_{m,n \in \mathbb{Z}} \mathbb{Q}(v,t)X^{m}Y^{n})\}$, where $X=K_{i}K'_{i},\ Y=E_{i}F_{i}+\frac{v^{-1}K_{i}+vK'_{i}}{(v-v^{-1})^{2}}$ and $k$ is in $(1)$.
\end{thm}

\begin{proof}
 For any $u \in U_{J_{i}}$, $u$ can be expressed as the sum

            $$u=\bigoplus_{x \in Q', y \in Q'} K_{x}K'_{y}L ,$$
where $L \in U_{v,t}(g_{_{i}})$.

 To determine the centre of $U_{J_{i}}$,  it is enough to find the central elements with the form $K_{x}K'_{y}L$ . That is to say, we have to check the equation
$$u(K_{x}K'_{y}L)=(K_{x}K'_{y}L) u, $$
for all $u \in U_{J_{i}}$.

(1) When $u=K_{j}$ for all $j \in I$.  $L$ has the form $L = \bigoplus_{a \in \mathbb{N} ; b,c \in \mathbb{Z}} E^{a}_{i}K_{bi}K'_{ci}F^{a}_{i}, i $ is the fixed one.

(2) When $u=K'_{j}$ for all $j \in I$. The result is the same as in case (1).

(3) When $u=E_{i}$, where $i$ is the fixed one. From the equation $E_{i}K_{x}K'_{y}L=K_{x}K'_{y}L E_{i}$, we may consider  $K_{x}K'_{y}(E_{i}L)=K_{x}K'_{y}(\theta L E_{i})$£¬ where $\theta$ is a constant  in $\mathbb{Q}(v,t)$ determined by $E_{i}$ and $K_{x}K'_{y}$. Assume $L$ has the form as in (1):

 $$L = \bigoplus_{a \in \mathbb{N} } E^{a}_{i}t_{a}F^{a}_{i},$$
here $t_{a} \in \mathbb{Q}(v,t)[K^{\pm}_{i},K'^{\pm}_{i}]$.

 Denote by $m$  the maximal number such that $E^{m}_{i}t_{m}F^{m}_{i}\neq 0$. Since we have $E_{i}L=\theta L E_{i}$, then
    $$E_{i}\bigoplus_{a \in \mathbb{N} } E^{a}_{i}t_{a}F^{a}_{i}=\theta \bigoplus_{a \in \mathbb{N} } E^{a}_{i}t_{a}F^{a}_{i} E_{i},$$
$$E^{m+1}_{i}t_{m}F^{m}_{i}+ E_{i}   \bigoplus_{a \in \mathbb{N},a<m } E^{a}_{i}t_{a}F^{a}_{i} = \theta E^{m}_{i}t_{m}F^{m}_{i}E_{i} +  \theta \bigoplus_{a \in \mathbb{N},a<m } E^{a}_{i}t_{a}F^{a}_{i} E_{i}.$$
It is obvious that $\theta E^{m}_{i}t_{m}F^{m}_{i}E_{i}$ is the only term on the right including  $E^{m+1}_{i}t_{m}F^{m}_{i}$. From the commutative relation of $U_{v,t}(g_{i})$, we conclude
$\theta \in v^{2n},$
where $n \in \mathbb{Z}$ determined by $E_{i}$ and $t_{m}.$

 The equation $E_{i}K_{x}K'_{y}L=K_{x}K'_{y}L E_{i}$ is equivalent to say $E_{i}L=\theta L E_{i}$, where $\theta$ determined by $K_{x}K'_{y}$ and $E_{i}$.  Form the above argument, we know $\theta$ is not an arbitrary element in $\mathbb{Q}(v,t)$ but belongs to $v^{2n}$. That is to say, $K_{x}K'_{y}$ should satisfy

 $$v^{2n}E_{i}K_{x}K'_{y}=K_{x}K'_{y} E_{i}, n \in \mathbb{Z}.$$

 The final problem is to find all  elements $L \in U_{v,t}(g_{i})$ satisfying

     $$ E_{i}L=v^{2n}L E_{i}, n \in \mathbb{Z}.$$

 For some  $k \in \mathbb{Z}$, set $\Gamma_{k}$ = $\{L| L \in U_{v,t}(g_{i}), E_{i}L=v^{2k}L E_{i} \}$.  In fact, for arbitrary $k \in \mathbb{Z}$,  we have $E_{i}K^{k}_{i}Z(U_{v,t}(g_{i}))=v^{2k}K^{k}_{i}Z(U_{v,t}(g_{i}))E_{i}$, and since $K_{i}$ is invertible, we conclude
  $$\Gamma_{k}=\{K^{k}_{i}(\Sigma_{m,n \in \mathbb{Z}} \mathbb{Q}(v,t)X^{m}Y^{n})\ |\ X=K_{i}K'_{i}, \ Y=E_{i}F_{i}+\frac{v^{-1}K_{i}+vK'_{i}}{(v-v^{-1})^{2}}\},$$
 for the fixed $i$ and $k$.

(4) When $u=F_{i}$,  where $i$ is the fixed one. The result is the same as in case (3).

\end{proof}

 Recall that \textrm{theorem 5.2} gives a criteria to justify when the extra central elements appear. Besides, to make the \textrm{proposition 5.4} more accurate, we calculate the centre of $U_{J_{i}}$ for each fixed $i \in I$ in \textrm{theorem 5.5}. Finally, we get the main theorem as follows.

\begin{thm} We have

   $(1)$\ If there exists $\eta \in Q$ such that $[\eta,i]-[i,\eta]=0$ for all $i \in  I$, then the centre $Z(U_{v,t})$ is isomorphic under $\xi$ to a subalgebra satisfying $\mathbb{Q}(v,t)[K_{\eta}'K_{\eta},K'^{-1}_{\eta}K^{-1}_{\eta}]\otimes(U^{0}_{\flat})^{W}\subseteq \xi(Z(U_{v,t}))\subseteq \bigcap_{i \in I}\xi(Z(U_{J_{i}}))$, where $Z(U_{J_{i}})$ is described in \textrm{theorem 5.5}.

    $(2)$ Otherwise  we have $\xi : Z(U_{v,t})\rightarrow (U^{0}_{\flat})^{W}$ is an isomorphism.\\
\end{thm}

  Recall in \textrm{theorem 5.2}  we show that the nonzero homogeneous central elements only exist in the degree $(\eta,\eta)$ with $\eta \in Q$ satisfying $[\eta,i]-[i,\eta]=0,\  \forall i \in I $. So far, there are two situations  described in the sense of Harish-Chandra homomorphism:

  (1)\ Denote by $Z_{0}$ the set of central elements with degree $(0,0)$, they are described via the Harish-Chandra homomorphism $\xi : Z_{0}\rightarrow (U^{0}_{\flat})^{W}$.

  (2)\ Denote by $Z_{ev}$ the set of central elements with nonzero degree $(2\eta,2\eta)$ with $\eta \in Q$ satisfying $[\eta,i]-[i,\eta]=0,\  \forall i \in I.$ Since $K'_{\eta}K_{\eta}$ is invertible central elements, $Z_{ev}$ are described via the Harish-Chandra homomorphism $\xi : Z_{ev}\rightarrow \mathbb{Q}(v,t)[K_{\eta}'K_{\eta},K'^{-1}_{\eta}K^{-1}_{\eta}]\otimes(U^{0}_{\flat})^{W}$.

Then it is natural  to ask the question : besides the central elements whose images of Harish-Chandra homomorphism are in $\mathbb{Q}(v,t)[K_{\eta}'K_{\eta},K'^{-1}_{\eta}K^{-1}_{\eta}]\otimes(U^{0}_{\flat})^{W}$, are there  any other extra central elements? Consider the $U_{v,t}(sl_{2})$ case, the Casimir element $EF+\frac{v^{-1}K+vK'}{(v-v^{-1})^{2}}$ hasn't been described in $\mathbb{Q}(v,t)[K_{\eta}'K_{\eta},K'^{-1}_{\eta}K^{-1}_{\eta}]\otimes(U^{0}_{\flat})^{W}$ under the Harish-Chandra homomorphism.  To answer the question in more general cases, we have the following theorem.

\begin{thm}
The homogeneous central elements exist in and only in grading $(\eta,\eta)$ for all $\eta \in Q$  satisfying $[\eta,i]-[i,\eta]=0$, $\forall$ $i \in I$.
\end{thm}

\begin{proof}

It is sufficient to show the existence.
Set $\{\alpha_{i}\}_{i \in I}$ as the set of simple roots. For one arbitrary  $\eta \in Q$ satisfying $[\eta,i]-[i,\eta]=0, \  \forall i \in I$. We express $\eta$ as the sum $x_{1}\alpha_{1}+x_{2}\alpha_{2}+...+x_{n}\alpha_{n}$ where $x_{1},x_{2},...,x_{n} \in \mathbb{Z}$ and $\alpha_{1},\alpha_{2},...,\alpha_{n}\in \{\alpha_{i}\}_{i \in I}$.

Recall the definition of multiplication $\ast$, it is equivelent to show the existence of central elements with grading $(\eta,\eta)$ under the multiplication $\ast$.

For finite type, a direct calculation shows  if $\eta=x_{1}\alpha_{1}+x_{2}\alpha_{2}+...+x_{n}\alpha_{n}$ satisfying $[\eta,i]-[i,\eta]=0$ $\forall$ $i \in I$, then there always exists  $y_{1},y_{2},...,y_{n} \in \mathbb{Z}$ such that for every $i = 1,2,...,n$, $x_{i}$ and $y_{i}$ have the same parity and $\nu=y_{1}\alpha_{1}+y_{2}\alpha_{2}+...+y_{n}\alpha_{n} \in Q \cap  2\Lambda$.

For any $\nu\in Q\cap2\Lambda$, from the conclusion of  Harish-Chandra isomorphism  in one parameter case, we know there exist one unique central element whose image of Harish-Chandra isomorphism is $(K_{\nu})^{W}$. We can denote  this central element as  $Z_{\nu}$  and   express it as $Z_{\nu}=\bigoplus_{\nu'\geq0}U^{-}_{\nu'}U^{0}U^{+}_{\nu'}$. And for every $\nu'>0$,  denote by $\{x_{\nu'}\}$ and $\{y_{\nu'}\}$ the basis of $U^{-}_{\nu'}$ and $U^{+}_{\nu'}$  respectively.
 Then $ Z_{\nu}$ can be expressed as
$$Z_{\nu}=\bigoplus_{\nu'\geq0}\bigoplus_{x_{\nu'} \in \{x_{\nu'}\}, y_{\nu'} \in \{y_{\nu'}\}} x_{\nu'}U^{0}y_{\nu'},$$
where $U^{0}$ denotes the polynomial algebra in Cartan part.

Fix some $\nu' \in Q$,  $x_{\nu'}$ and $y_{\nu'}$, then write $x_{\nu'}U^{0}y_{\nu'}$ as  $\bigoplus_{\varsigma\in Q} \mathbb{Q}(v,t)x_{\nu'}K_{\varsigma}y_{\nu'}$. Corresponding to every $x_{\nu'}U^{0}y_{\nu'}$, we define one term $x_{\nu'}\tilde{U}^{0}y_{\nu'}$ in $((U_{v,t},(I,\cdot)),\ast)$
as $\bigoplus_{\varsigma\in Q} \mathbb{Q}(v,t)x_{\nu'}K_{\frac{\eta-\nu'+\varsigma}{2}}K'_{\frac{\eta-\nu'-\varsigma}{2}}y_{\nu'}$.

 Then we can define one element $\tilde{Z}_{\nu}$ as

$$\tilde{Z}_{\nu}=\bigoplus_{\nu'\geq0}\bigoplus_{x_{\nu'} \in \{x_{\nu'}\}, y_{\nu'} \in \{y_{\nu'}\}} \mathbb{Q}(v)x_{\nu'}\tilde{U}^{0}y_{\nu'},$$

 comparing the communicative relation,  since the only difference exists in Cartan part, we know $\tilde{Z}_{\nu}$ is in the centre of $U_{v,t}(I,\cdot)$.

\end{proof}


\end{document}